\documentclass[a4paper, 11pt]{amsart}
\usepackage{geometry}
\usepackage{amssymb}
\usepackage[all]{xy}
\usepackage{hyperref,aliascnt}
\usepackage{enumerate}
\usepackage{stmaryrd}
\usepackage{mathtools}
\usepackage{xcolor}

\numberwithin{equation}{section}

\newtheorem{lma}{Lemma}[section]

\newaliascnt{thmCt}{lma}
\newtheorem{thm}[thmCt]{Theorem}
\aliascntresetthe{thmCt}

\newaliascnt{corCt}{lma}
\newtheorem{cor}[corCt]{Corollary}
\aliascntresetthe{corCt}

\newaliascnt{prpCt}{lma}
\newtheorem{prp}[prpCt]{Proposition}
\aliascntresetthe{prpCt}

\newtheorem*{thm*}{Theorem}
\newtheorem*{cor*}{Corollary}
\newtheorem*{prp*}{Proposition}

\theoremstyle{definition}

\newaliascnt{pgrCt}{lma}
\newtheorem{pgr}[pgrCt]{}
\aliascntresetthe{pgrCt}

\newaliascnt{dfnCt}{lma}
\newtheorem{dfn}[dfnCt]{Definition}
\aliascntresetthe{dfnCt}

\newaliascnt{rmkCt}{lma}
\newtheorem{rmk}[rmkCt]{Remark}
\aliascntresetthe{rmkCt}

\newaliascnt{qstCt}{lma}

\aliascntresetthe{qstCt}

\newaliascnt{exaCt}{lma}
\newtheorem{exa}[exaCt]{Example}
\aliascntresetthe{exaCt}

\newaliascnt{ntnCt}{lma}
\newtheorem{ntn}[ntnCt]{Notation}
\aliascntresetthe{ntnCt}

\newcommand{\NN}{\mathbb{N}}
\newcommand{\NNbar}{\overline{\mathbb{N}}}
\newcommand{\ca}{$C^*$-al\-ge\-bra}

\newcommand{\axiomO}[1]{(O#1)}
\newcommand{\CuSgp}{$\CatCu$-sem\-i\-group}
\newcommand{\CuMor}{$\CatCu$-mor\-phism}

\DeclareMathOperator{\Cu}{Cu}
\DeclareMathOperator{\Lsc}{Lsc}
\DeclareMathOperator{\mesh}{mesh}
\DeclareMathOperator{\diam}{diam}

\newcounter{theoremintro}

\newtheorem{thmIntro}[theoremintro]{Theorem}

\newcommand{\CatTop}{\mathrm{Top}}
\newcommand{\CatCu}{\ensuremath{\mathrm{Cu}}}

\newcommand{\andSep}{\,\,\,\text{ and }\,\,\,}

\DeclareMathOperator{\Int}{Int}

\newcommand{\Msection}[2]{\section{\texorpdfstring{#1}{#2}}}

\title{The Cuntz semigroup of unital commutative AI-algebras}

\author{Eduard Vilalta}
\address{E.~Vilalta, Departament de Matem\`{a}tiques,
Universitat Aut\`{o}noma de Barcelona,
08193 Bellaterra, Barcelona, Spain}
\email{evilalta@mat.uab.cat}
\thanks{
The author was partially supported by MINECO (grant No.\ PRE2018-083419 and No.\ MTM2017-83487-P), and by the Comissionat per Universitats i Recerca de la Generalitat de Catalunya (grant No.\ 2017SGR01725).
}
\subjclass[2010]%
{Primary
46L05, % General theory of C*-algebras
46L85. % Noncommutative topology
}
\keywords{$C^*$-algebras, Cuntz semigroups, AI-algebras}
\date{\today}

\begin{document}

%==========================================================================================
\begin{abstract}
We provide an abstract characterization for the Cuntz semigroup of unital commutative AI-algebras, as well as a characterization for abstract Cuntz semigroups of the form $\Lsc (X,\overline{\NN})$ for some $T_1$-space $X$. In our investigations, we also uncover new properties that the Cuntz semigroup of all AI-algebras satisfies.
\end{abstract}

\maketitle

%==========================================================================================
%===========================================

\section{Introduction}

The celebrated Effros-Handelman-Shen theorem  \cite[Theorem~2.2]{EffHanShe80} characterizes when a countable ordered abelian group $G$ is order isomorphic to the ordered $K_0$-group of an AF-algebra. More explicitly, it states that $G$ is unperforated and has the Riesz interpolation property if and only if $G$ is order isomorphic to the $K_0$-group of such a \ca{}.

In analogy to the definition of an AF-algebra, a \ca{} $A$ is said to be an AI-algebra if $A$ is $*$-isomorphic to an inductive limit whose building blocks have the form $C[0,1]\otimes F_n$ with $F_n$ finite dimensional for every $n$. In the unital commutative setting, an AI-algebra is of the form $C(X)$, with $X$ homeomorphic to an inverse limit of (possibly increasing) finite  disjoint unions of unit intervals.

In this paper we use the Cuntz semigroup of a \ca{}, a refinement of $K_0$ introduced by Cuntz in \cite{Cun78DimFct} and used succesfully in the classification of not necessarily simple \ca{s}; see, for example, \cite{CiupElli08}, \cite{CiupElliSant11}, \cite{RobSan10}, \cite{Rob12LimitsNCCW}. We provide an abstract characterization for the Cuntz semigroup of unital commutative AI-algebras and introduce new properties that the Cuntz semigroup of every AI-algebra satisfies. This settles the range problem for the Cuntz semigroup of this class of commutative \ca{s};  the corresponding problem in the setting of general AI-algebras was posed during the 2018 mini-workshop on the Cuntz semigroup in Houston, and is studied in \cite{Vil21arX:LocalChar}.

Abstracting some of the properties that the Cuntz semigroup of a \ca{} always satisfies, the subcategory $\Cu$ of partially ordered monoids was introduced in \cite{CowEllIva08CuInv}. This subcategory, whose objects are often called \CuSgp{s}, contains the Cuntz semigroup of all \ca{s}, and has been studied extensively in
\cite{AraPerTom11Cu},
\cite{AntoPereThie18},
\cite{AntPerThi20:AbsBivariantCu} and \cite{AntPerRobThi21:CuntzSR1} among others. A relevant family in $\Cu$ is that of \emph{countably based} \CuSgp{s} satisfying \emph{\axiomO{5}} (see \autoref{pgr:AdditionalAxioms}), which contains the Cuntz semigroups of all separable \ca{s}; see \cite{RorWin10ZRevisited} and \cite{AntPerSan11PullbacksCu}.

The range problem for this class of algebras consists of finding a list of properties that a \CuSgp{} $S$ satisfies if and only if $S$ is $\Cu$-isomorphic to the Cuntz semigroup of a unital commutative AI-algebra. In order to do this, we first  investigate when a compact metric space $X$ is such that $C(X)$ is an AI-algebra, that is to say, we analyze when $X$ is homemorphic to an inverse limit of finite  disjoint unions of unit intervals. To this end, and in analogy to the definition of a  \emph{chainable} continuum (see, for example, \cite[Chapter~12]{Nadl92}), we introduce \emph{almost chainable} and \emph{generalized arc-like} spaces; see \autoref{dfn:AlmCha} and \autoref{dfn:GenArcLik} respectively. We prove:

\begin{thmIntro}\label{thmIntro:AchIffCom}[\ref{achain_iff_comAI}]
  Let $X$ be a compact metric space. The following are equivalent:
 \begin{enumerate}[(i)]
  \item $X$ is almost chainable.
  \item $X$ is a generalized arc-like space.
  \item $X$ is an inverse limit of finite disjoint copies of unit intervals.
  \item $C(X)$ is an AI-algebra.
 \end{enumerate}
\end{thmIntro}

The dimension of the spaces appearing in Theorem \ref{thmIntro:AchIffCom} is at most one, and in this case the Cuntz semigroup of $C(X)$ is isomorphic to the semigroup $\Lsc (X,\overline{\NN})$ of lower semicontinuous functions from $X$ to $\{ 0,1,\cdots ,\infty \}$ (see, e.g. \cite{Rob13CuSpDim2}). Thus, the next step in our approach is to characterize those \CuSgp{s} of the form $\Lsc (X,\overline{\NN})$ for some $T_1$-space $X$. In \autoref{Lsc_like_Cu_sem} we define the notion of an $\Lsc$-like \CuSgp{}. Given such a semigroup $S$, we prove in \autoref{Topological_Lsc_like} that $S$ has an associated $T_1$-topological space $X_S$ (\autoref{pgr:TopSpaCuSgp}) and that many properties defined for \CuSgp{s} have a topological counterpart whenever the semigroup is $\Lsc$-like (see \autoref{prp:TopProp}). For example, an $\Lsc$-like \CuSgp{} $S$ has a \emph{compact} order unit (in the sense of \autoref{pgr:AdditionalAxioms} below) if and only if $X_S$ is countably compact. In \autoref{LscXS_iff_Lsc_like}, we show that $S$ is $\Lsc$-like if and only if it is $\Cu$-isomorphic to $\Lsc (X_S,\NNbar )$.

Using this characterization, together with the notion of covering dimension for \CuSgp{s} introduced in \cite{ThiVil21arX:DimCu}, we obtain the following result.
\begin{thmIntro}[\ref{thm:dimChar}]
 Let $S$ be a \CuSgp{} satisfying \axiomO{5} and let $n\in\NN\cup\{ \infty\}$. Then, $S$ is $\Cu$-isomorphic to $\Lsc (X,\NNbar )$ with $X$ a compact metric space such that $\dim (X)= n$ if and only if $S$ is $\Lsc$-like, countably based, satisfies \axiomO{5}, has a compact order unit, and $\dim (S)=n$.
 
 In particular, a \CuSgp{} $S$ is $\Cu$-isomorphic to the Cuntz semigroup of the \ca{} $C(X)$ with $X$ compact metric and $\dim (X)\leq 1$ if and only if $S$ is $\Lsc$-like, countably based, satisfies \axiomO{5}, has a compact order unit, and $\dim (S)\leq 1$.
\end{thmIntro}

With these results at hand, we introduce in \autoref{Sct_proper} notions of \emph{chainability} for \CuSgp{s}. These conditions aim at modelling, at the level of abstract Cuntz semigroups, the concepts of cover and chain for a topological space. We prove:

\begin{thmIntro}[{\ref{almost_chainable_Cu}}]
 Let $S$ be a \CuSgp{}. Then, $S$ is $\Cu$-isomorphic to the Cuntz semigroup of a unital commutative AI-algebra if and only if $S$ is countably based, $\Lsc$-like,  weakly chainable, has a compact order unit, and satisfies \axiomO{5}.
\end{thmIntro}

Finally, in \autoref{Sect_Proper}, we generalize some of the properties used in \autoref{almost_chainable_Cu} and show that these generalizations are satisfied for the Cuntz semigroup of every AI-algebra; see \autoref{thm:AIProperties}.

In \cite{Vil21arX:LocalChar} we continue the study of Cuntz semigroups of AI-algebras and develop tools towards a similar characterization of \autoref{almost_chainable_Cu} in the general case.
\vspace{0.1cm}

\textbf{Acknowledgments.} This paper constitutes a part of the Ph.D. dissertation of the author. He is indebted to his advisor  Francesc Perera for his encouragement and insights. He also wishes to thank Hannes Thiel for his comments and suggestions on a first draft which greatly improved the paper.
%===========================================

\section{Preliminaries}

Given a \ca{} $A$ and two positive elements $a,b\in A$, recall that $a$ is \emph{Cuntz subequivalent} to $b$, and write $a\precsim b$, if there exists a sequence $(r_n)_n$ in $A$ such that $a=\lim_n r_n b r_n^*$. Moreover, if $a\precsim b$ and $b\precsim a$, we say that $a$ and $b$ are \emph{Cuntz equivalent}, in symbols $a\sim b$.

The \emph{Cuntz semigroup} of $A$, denoted by $\Cu (A)$, is defined as the quotient $(A\otimes \mathcal{K})_+/{\sim }$, where we denote by $[a]$ the class of an element $a\in (A\otimes \mathcal{K})_+$. Equipped with the order induced by $\precsim$ and the addition induced by $[a]+[b]=\left[\begin{psmallmatrix}a&0\\ 0&b\end{psmallmatrix}\right]$, the Cuntz semigroup becomes a positively ordered monoid; see \cite{Cun78DimFct}, \cite{CowEllIva08CuInv}.

\begin{pgr}
 Given two elements $x,y$ in a partially ordered set, we say that $x$ is \emph{way-below} $y$, and write $x\ll y$, if for every increasing sequence $(z_n)_n$ whose supremum exists and is greater than or equal to $y$ there exists $n\in\NN$ such that $x\leq z_n$.
 
 In \cite{CowEllIva08CuInv} it was shown that the Cuntz semigroup $S$ of every \ca{} always satisfies the following properties:
 \begin{itemize}
 \item[\axiomO{1}] Every increasing sequence in $S$ has a supremum.
 \item[\axiomO{2}] Every element in $S$ can be written as the supremum of an $\ll$-increasing sequence.
 \item[\axiomO{3}] For every $x'\ll x$ and $y'\ll y$, we have $x'+y'\ll x+y$.
 \item[\axiomO{4}] For every pair of increasing sequences $(x_n)_n$ and $(y_n)_n$ we have $\sup_n x_n +\sup_n y_n = \sup_n (x_n+y_n)_n$.
\end{itemize}
\end{pgr}

\begin{pgr}
In a more abstract setting, we say that a positively ordered monoid is a \emph{\CuSgp{}} if it  satisfies \axiomO{1}-\axiomO{4}.

We also say that a map between two \CuSgp{s} is a \emph{\CuMor{}} if it is a positively ordered monoid morphism that preserves suprema of increasing sequences and the way-below relation. As proved in \cite{CowEllIva08CuInv}, every *-homomorphism $\varphi$ between \ca{s} induces a \CuMor{} $\Cu (\varphi)$ between their Cuntz semigroups.

Thus, one can consider the subcategory $\Cu$ of the category of positively ordered monoid as the category with \CuSgp{s} and \CuMor{s} as its objects and morphisms respectively. By the results from \cite{CowEllIva08CuInv}, the assignment $\Cu\colon C^*\to \Cu$ is functorial.
 
 Moreover, we know from \cite[Corollary~3.2.9]{AntoPereThie18}  that the category $\Cu$ has arbitrary inductive limits and that the functor $\Cu$ is arbitrarily continuous; see also \cite{CowEllIva08CuInv}.
\end{pgr}

\begin{pgr}\label{pgr:AdditionalAxioms}
A \CuSgp{} is said to have \emph{weak cancellation} if $x\ll y$ whenever $x+z\ll y+z$ for some element $z$. Given a \emph{compact} element $z$, that is, an element such that $z\ll z$, weak cancellation implies that $x\leq y$ whenever $x+z\leq y+z$. Stable rank one \ca{s} have weakly cancellative Cuntz semigroups by \cite[Theorem~4.3]{RorWin10ZRevisited}.

It was also proven in \cite{RorWin10ZRevisited} that the Cuntz semigroup of a \ca{} always satisfies the following property: For every $x'\ll x\leq z$, there exists $c$ such that 
 \[
  x'+c\leq z\leq x+c.
 \]

 Moreover, it was shown in \cite[Proposition~4.6]{AntoPereThie18} that a  stronger property termed \axiomO{5} was also  satisfied for the Cuntz semigroup of any \ca{}. Under the assumption of weak cancellation, these two properties coincide.
 
 We will say that a \CuSgp{} is \emph{countably based} if there exists a countable subset such that every element in the semigroup can be written as the supremum of an increasing sequence of elements in the subset. It follows from  \cite[Lemma~1.3]{AntPerSan11PullbacksCu} that Cuntz semigroups of separable \ca{s} are countably based.
\end{pgr}

\begin{pgr}\label{pgr:Thom}
Recall that a \ca{} $A$ is an AI-algebra if it is *-isomorphic to an inductive limit of the form $\lim_n C[0,1]\otimes F_n$ with $F_n$ finite dimensional for every $n$.

 By \cite{Thom92}, every unital commutative AI-algebra is *-isomorphic to an inductive limit of the form
\begin{equation*}
\xymatrix{
 C([0,1])^{n_{1}} \ar[r]^-{\varphi_{2,1}} & C([0,1])^{n_{2}} \ar[r]^-{\varphi_{3,2}} & 
 C([0,1])^{n_{3}} \ar[r]^-{\varphi_{4,3}} & 
 \cdots
 }
\end{equation*}
with $n_{i}\leq n_{i+1}$ for each $i$ and where all the homomorphisms are in standard form. That is to say, for every $i$ we have
\begin{equation*}
 \pi_{j,i+1}\varphi_{i+1,i}(f_{1},\cdots ,f_{n_{i}}) = f_{r}\sigma_{j,i},
\end{equation*}
where $\pi_{j,i+1}\colon C([0,1])^{n_{i+1}}\rightarrow C([0,1])$ is the $j$-th projection map, $\sigma_{j,i}\colon [0,1]\rightarrow [0,1]$ is a continuous function and $r\leq n_{i}$.

Thus, every unital commutative AI-algebra is isomorphic to $C(X)$, where $X$ is an inverse limit of (possibly increasing) finite disjoint unions of unit intervals. Conversely, it is easy to see that $C(X)$ is an AI-algebra for such an inverse limit $X$.

Since the space $X$ above will always have dimension at most one, we know by \cite[Theorem~1.1]{Rob13CuSpDim2} that $\Cu (C(X))\cong\Lsc (X,\overline {\NN})$.

Also note that, with the above notation, if there exists some $i$ such that $n_{i}=n_{j}$ for every $j\geq i$, we can write $A\cong \oplus_{k=1}^{n_{i}} C(X_{k})$ with $X_{k}$ an inverse limit of unit intervals.
\end{pgr}
%===========================================

\section{Chainable and almost chainable spaces}

In this section we will prove that a compact metric space $X$ is homeomorphic to an inverse limit of finite disjoint copies of unit  intervals (i.e. $C(X)$ is an AI-algebra) if and only if $X$ satisfies an abstract property, which we call almost chainability; see \autoref{achain_iff_comAI}. This is done in analogy to \cite[Chapter 12]{Nadl92}, where it is shown that a continuum is homeomorphic to the inverse limit of unit intervals if and only if it is chainable, as defined below.

\begin{pgr}
Recall that a \emph{compactum} is a compact metric space, and that a \emph{continuum} is a connected compactum. As in \cite[Chapter~12]{Nadl92}, a \emph{chain} in a continuum $X$ will be a finite non-empty indexed collection $C=\{C_{1},\cdots , C_{k}\}$ of open subsets of $X$ such that 
\[
 C_{i}\cap C_{j}\neq\emptyset \text{ if and only if }\vert i-j\vert\leq 1.
\]

The \emph{mesh} of a chain $C=\{C_{1},\cdots , C_{k}\}$, in symbols $\mesh (C)$, is defined as
\[
 \mesh (C)=\max\{ \diam (C_{i})\}.
\]

A chain of mesh less than $\epsilon$ is called an \emph{$\epsilon$-chain}.
\end{pgr}
%===========================================

Following \cite[Chapter~12]{Nadl92}, one can now define chainable continua.

\begin{dfn}\label{dfn_chainable}
 A continuum $X$ is said to be \emph{chainable} if for every positive $\epsilon$ there exists an $\epsilon$-chain covering $X$.
 
 We will say that a compactum is \emph{piecewise chainable} if it can be written as the finite disjoint union of closed chainable subspaces. 
\end{dfn}
%===========================================

Chainability may also be defined in a general setting, and one can check that the previous and following definitions coincide whenever $X$ is a continuum (see \autoref{Eq_metric_topo}).

\begin{dfn}
 A topological space $X$ is said to be \emph{topologically chainable} if any finite open cover of $X$ can be refined by a chain, that is, an open cover $C_{1},\cdots ,C_{k}$ such that $C_{i}\cap C_{j}\neq\emptyset$ if and only if $\vert i-j\vert\leq 1$.
\end{dfn}
%===========================================
\begin{rmk}\label{rmk:ChainImplDim1}
 Note that both of the abovementioned definitions imply that the dimension of the space is at most one, and that the space is connected whenever it is compact.
\end{rmk}

%===========================================
The following proposition shows why we consider chainable continua. More on chainable spaces can be found in \cite[Chapter 12]{Nadl92}.

\begin{prp}
 A compactum is an inverse limit of unit intervals if and only if it is chainable.
\end{prp}
\begin{proof}
 Let $X$ be the inverse limit of unit intervals. Since the unit interval is compact and connected, $X$ is connected and, consequently, a continuum.
 
 If $X$ is degenerate (i.e a point), it is clearly chainable, so we may assume otherwise. 
 
 If $X$ is non-degenerate, then \cite[Theorems 12.11,12.19]{Nadl92} and the comments following Theorem 12.19 of \cite{Nadl92} imply that $X$ is chainable.
 \vspace{0.1cm}
 
 Conversely, if $X$ is chainable, it can either be degenerate (in which case we are done) or non-degenerate. By \cite[Theorem 12.11]{Nadl92}, a non-degenerate chainable continuum is an inverse limit of unit intervals, so the result follows.
\end{proof}
%===========================================

\begin{dfn}\label{dfn:PieceWiseChain}
 A unital AI-algebra $A$ will be said to be \emph{block-stable} if it is isomorphic to $C(X)$ with $X$ a compact metric piecewise chainable space. That is, if $A\cong \oplus_{k=1}^{n}C(X_{k})$ with $X_{k}$ a chainable continuum for each $k$.
\end{dfn}
%===========================================

\subsection{Generalized arc-like spaces}

We now generalize the previous results and characterize the topological spaces arising from unital commutative AI-algebras in terms of a weaker property:

\begin{pgr}
Let $X$ be a compactum. In analogy with the definition of chains, an \emph{almost chain} in $X$ will be a finite non-empty indexed collection $C=\{ C_{1},\cdots , C_{k} \}$ of open subsets of $X$ such that
\[
 C_{i}\cap C_{j}= \emptyset \text{ whenever } \vert i-j\vert \geq 2. 
\]

The \emph{mesh} of an almost chain will be $\mesh (C)=\max\{\diam (C_{i}) \}$, and an almost chain of mesh less than $\epsilon$ will be called an \emph{$\epsilon$-almost chain}.
\end{pgr}
%===========================================

\begin{dfn}\label{dfn:AlmCha}
 A compact metric space $X$ will be said to be \emph{almost chainable} if, for every $\epsilon>0$, there exists an $\epsilon$-almost chain covering $X$.
\end{dfn}

\begin{dfn}\label{dfn:TopAlmCha}
 A topological space $X$ will be said to be \emph{topologically almost chainable} if any finite open cover of $X$ can be refined by an almost chain, that is, an open refinement $C_{1},\cdots ,C_{k}$ such that $C_{i}\cap C_{j}= \emptyset$ whenever $\vert i-j\vert\geq 2$.
\end{dfn}
%============================================

\begin{lma}\label{Eq_metric_topo}
  \autoref{dfn:AlmCha} and \autoref{dfn:TopAlmCha} above coincide whenever $X$ is a  compact metric space.
  
  The same proof can be applied to  chainability and topological chainability.
 \end{lma}
 \begin{proof}
  If $X$ is topologically almost chainable, let $\epsilon >0$ and take an open cover of $\epsilon$-balls in $X$. Since $X$ is compact, there exist finitely many points $x_{1},\cdots ,x_{n}$ such that their $\epsilon$-balls cover $X$.
  
  By topological almost chainability, this finite open cover can be refined by open subsets $C_{1},\cdots ,C_{k}$ such that $C_{i}\cap C_{j}=\emptyset$ whenever $\vert i-j\vert\geq 2$. Since $C_{i}$ is contained in some $\epsilon$-ball, it follows that $C_{1},\cdots ,C_{k}$ is an $\epsilon$-almost chain.
  
  Conversely, assume that $X$ is almost chainable and take a (finite) open cover $U_{1},\cdots ,U_{n}$.
  
  Since $X$ is a compact metric space, the cover has a non-zero Lebesgue number $\delta$. That is, every subset of $X$ having diameter less than $\delta$ is contained in some $U_{j}$.
  
  Set $\epsilon <\delta$ and consider an $\epsilon$-almost chain $C=\{ C_{1},\cdots ,C_{k} \}$ covering $X$. Since $\diam (C_{i})\leq \epsilon <\delta$ for every $i$, it follows that each $C_{i}$ is contained in some $U_{j}$. Consequently, $C_{1},\cdots ,C_{k}$ is an open refinement of $U_{1},\cdots , U_{n}$ with the required property.
 \end{proof}
%===========================================

We now define the notion of $(\epsilon , \delta)$-maps and generalized arc-like spaces. This is done in analogy with \cite[Theorem 12.11]{Nadl92}.

Recall that a continuous map $f\colon X\rightarrow Y$ is an \emph{$\epsilon$-map} if $\diam (f^{-1}( y) )\leq \epsilon$ for every $y\in Y$, where $\diam (\emptyset )=0$ by definition.

\begin{dfn}
 Given $\epsilon,\delta>0$, we will say that a continuous map $f\colon X\rightarrow Y$ is an \emph{$(\epsilon, \delta)$-map} if $\diam (f^{-1}(Z))< \epsilon$ for every $Z\subset Y$ with $\diam (Z)<\delta$.
\end{dfn}
%===========================================

Recall also that a continuum $X$ is \emph{arc-like} if for every $\epsilon >0$ there exists an $\epsilon$-map from $X$ onto $[0,1]$; see \cite[Definition~2.12]{Nadl92}.

\begin{dfn}\label{dfn:GenArcLik}
 A compactum is said to be a \emph{generalized arc-like space} if for every $\epsilon >0$ there exists $\delta >0$ and an $(\epsilon ,\delta)$-map $f\colon X\rightarrow I_{1}\sqcup\cdots\sqcup I_{n}$ with $I_{j}=[0,1]$ for each $j$.
\end{dfn}
%===========================================

\begin{rmk}\label{rmk:OneInt}
 Given a finite disjoint union of unit intervals $I_{1}\sqcup\cdots\sqcup I_{n}$ and some $\delta >0$, one can clearly construct a $(\delta, \delta ')$-map $r\colon I_{1}\sqcup\cdots\sqcup I_{n}\rightarrow [0,1]$  (by simply rescaling $I_{1}\sqcup\cdots\sqcup I_{n}$ until it fits in $[0,1]$).
 
 Thus, a compactum is a generalized arc-like space if and only if for every $\epsilon >0$ there exists an $(\epsilon ,\delta)$-map $f\colon X\rightarrow [0,1]$ for some $\delta >0$.
\end{rmk}
%===========================================

\begin{lma}
 Inverse limits of finite disjoint unions of unit intervals are generalized arc-like spaces.
\end{lma}
\begin{proof}
Let $X$ be an inverse limit of finite disjoint unions of unit intervals, and let $( [0,1]\sqcup\cdots\sqcup [0,1] ,f_{i,j} )$ be its associated inverse system.

Recall that the metric on $X$ is defined to be 
\[
 \text{d}((x_{i})_{i=1}^{\infty} , (y_{i})_{i=1}^{\infty})=\sum_{i=1}^{\infty} \frac{1}{2^{i}} \frac{\text{d}_{i}(x_{i},y_{i})}{1+\text{d}_{i}(x_{i},y_{i})},
\]
where $\text{d}_{i}$ is the distance in the i-th component of the inverse system.

Also recall that the distance between two points $x,y$ in $[0,1]\sqcup\cdots\sqcup [0,1]$ is either the usual distance if $x,y$ belong to the same connected component or $2$ if they do not.\vspace{0.1cm}

 Given $\epsilon >0$, let $n$ be such that $\sum_{i> n} 2/2^{-i}<\epsilon/2$.
 
 Then, since for every fixed $n$ the maps $f_{i,n}$ are uniformly continuous for every $i\leq n$, there exists $\delta >0$  such that
 \[
 \text{d}_{i}(f_{i,n}(x),f_{i,n}(y))\leq\epsilon /2n
 \]
 whenever $\text{d}_{n}(x,y)\leq\delta$.
 
 Let $\pi_{n}\colon X\rightarrow [0,1]\sqcup\cdots\sqcup [0,1]$ be the $n$-th canonical projection map, and take $Z\subset\pi_{n}(X)$ with $\diam_{n} (Z)\leq \delta$,
 
 Take $x,y\in \pi_{n}^{-1}(Z)$. Then, we have
 \[
  \text{d}(x,y)=\sum_{i=1}^{n} \frac{1}{2^{i}}\frac{\text{d}_i(f_{i,n}(x_{n}),f_{i,n}(y_{n}))}{1+ \text{d}_i(f_{i,n}(x_{n}),f_{i,n}(y_{n}))} + \sum_{i>n}\frac{1}{2^{i}} \frac{\text{d}_{i}(x_{i},y_{i})}{1+\text{d}_{i}(x_{i},y_{i})}\leq n\frac{\epsilon}{2n} +
  \sum_{i>n}\frac{1}{2^{i}} 2\leq
  \epsilon ,
 \]
since any pair of elements is at distance at most $2$.

This shows that $\diam (\pi_{n}^{-1}(Z)) \leq \epsilon$, and so $\pi_{n}$ is an $(\epsilon ,\delta)$-map, as required.
\end{proof}
%===========================================

\begin{prp}\label{eq_AC}
 A compactum $X$ is a generalized arc-like space if and only if it is almost chainable.
\end{prp}
\begin{proof}
We follow the proof of \cite[Theorem 12.11]{Nadl92} while making some minor adjustments.

Assume first that $X$ is almost chainable and take $\epsilon >0$. Let $C$ be an $\epsilon /2$-almost chain covering $X$, and decompose it as $C=C_{1}\sqcup\cdots\sqcup C_{r}$ with
\[
 C_{1}=\{ C_{i,1} \}_{i} ,\cdots , C_{r}=\{ C_{i,r} \}_{i}
\]
$\epsilon /2$-chains in $X$. Note that for any $i,l$ we have $C_{i,j}\cap C_{l,k}=\emptyset$ whenever $j\neq k$ and 
\[
 X=\bigsqcup_{j} \bigcup_{i} C_{i,j}
\]
because $C$ covers $X$.

Take $s\leq r$. If $\vert C_{s}\vert \leq 2$, consider the map $f_{s}\colon \cup_{i} C_{i,s}\rightarrow [0,1]$ sending every element to $0$. Note that $\diam (f_{s}^{-1}(Z))<\epsilon$ whenever $\diam (Z)<1$.

If $\vert C_{s}\vert \geq 3$, use the techniques in \cite[Theorem 12.11]{Nadl92} to obtain an $(\epsilon,\delta_{s})$-map $f_{s}\colon \cup_{i} C_{i,s}\rightarrow [0,1]$.

Now define $f=f_{1}\sqcup\cdots\sqcup f_{r}\colon X\rightarrow I_{1}\sqcup\cdots\sqcup I_{r}$, where $I_{l}=[0,1]$ for every $l$. Note that this is clearly continuous.

Setting $\delta< \min\{1, \delta_{s}\}$ one gets that, if the diameter of $Z\subset I_{1}\sqcup\cdots\sqcup I_{r}$ is less than $\delta$, $Z$ is a subset of some $I_{l}$ (since the distance between disjoint components is $2$).

This implies that $f^{-1}(Z)=f^{-1}_{l}(Z)$ has diameter at most $\epsilon$, as required.
 
 Conversely, if $X$ is a generalized arc-like, then given any $\epsilon$ there exists $\delta$ with $f\colon X\rightarrow I_{1}\sqcup\cdots\sqcup I_{r}$ an $(\epsilon ,\delta)$-map.
 
 For each $I_{l}$, consider a $\delta$-chain $C_{l}$. The inverse image of these chains through $f$ gives the desired $\epsilon$-almost chain for $X$.
\end{proof}
%===========================================

\begin{cor}\label{necessity}
 Inverse limits of disjoint unions of unit intervals are almost chainable. In particular, each connected component is arc-like.
\end{cor}
\begin{proof}
 This follows from our previous two results. Note, however, that only one implication of \autoref{eq_AC} has been used.
\end{proof}
%===========================================

We will now show that the converse of \autoref{necessity} also holds, thus obtaining a characterization of inverse limits of finite disjoint copies of unit intervals; see \autoref{achain_iff_comAI}.

We first recall the following result from \cite{Freu37}, even though this particular formulation is from 
\cite[Lemma 2.14, page 184]{Thom92}.

\begin{lma}\label{lma:Freu}
 Let $C$ be a closed non-empty subset of $[0,1]$ and let $\epsilon >0$. Then, there exist a subset $Y\subset C$ which is a finite disjoint union of closed intervals (possibly degenerate) and a map $\alpha \colon C\rightarrow Y$ such that $\alpha (y)=y$ for every $y\in Y$ and $\vert \alpha (x)-x\vert <\epsilon$ for every $x\in C$.
\end{lma}

\begin{rmk}
 With the previous notation, note that $\alpha$ is an onto $2 \epsilon$-map from $C$ to $Y$.
 
 Indeed, for any $y\in Y$ and for any $x\in C$ with $\alpha (x)=y$, we must have $\vert x-y\vert <\epsilon$. Since $y\in \alpha^{-1}(y)$, it follows that $\diam (\alpha^{-1}(y))\leq 
 \diam (B_{\epsilon}(y))<2 \epsilon$.
\end{rmk}
%===========================================

From now on, by a \emph{closed interval} we will mean a possibly degenerate closed interval (that is to say, either a point or a non-degenerate closed interval).

\begin{pgr}\label{pgr:epsilon_Map_reduc}
Given any $(\epsilon ,\delta)$-map $f\colon X\rightarrow [0,1]$, consider the induced onto $\epsilon$-map $f\colon X\rightarrow \text{Im}(f)$.

Since $X$ is compact, so is $\text{Im}(f)$. Using \autoref{lma:Freu}, we can find an onto $\delta$-map $\alpha\colon\text{Im}(f)\rightarrow Y$ with $Y$ the finite disjoint union of  closed intervals and, consequently, we get an onto $\epsilon$-map $\alpha f\colon X\rightarrow Y$.

This shows that given any generalized arc-like compactum $X$ (equivalently, any almost chainable compactum, by \autoref{eq_AC}) and any $\epsilon >0$, there exists an onto $\epsilon$-map $f\colon X\rightarrow Y$ with $Y$ a  finite disjoint union of closed intervals. Indeed, given any $\epsilon >0$ we know by \autoref{dfn:GenArcLik} and \autoref{rmk:OneInt} that there exists an $(\epsilon ,\delta )$-map $f\colon X\to Y$. The conclusion now follows from the previous argument.
\end{pgr}

Using \autoref{pgr:epsilon_Map_reduc} above, we will see that $C(X)$ is a commutative AI-algebra for any generalized arc-like compactum $X$. Since we already know that $X$ is almost chainable whenever $C(X)$ is an AI-algebra (see \autoref{necessity}), this will show that $X$ is almost chainable if and only if $C(X)$ is a commutative AI-algebra.
\vspace{0.1cm}
%===========================================

The following lemma is a natural generalization of \cite[Lemma 12.17]{Nadl92}. We follow both the structure and the notation of the proof of said lemma.

\begin{lma}
 Let $X$ be a compactum, let $g_{1}\colon X\rightarrow Y_{1}$ be an onto continuous map with $Y_{1}$ a finite disjoint union of closed intervals, and let $\eta >0$.
 
 Then, there exists $\epsilon >0$ such that, for any onto $\epsilon$-map $g_{2}\colon X\rightarrow Y_{2}$ with $Y_{2}$ a finite disjoint union of closed intervals, there exists a continuous map $\varphi\colon Y_{2}\rightarrow Y_{1}$ such that $\vert g_{1}(x)-\varphi g_{2}(x)\vert <\eta$ for every $x\in X$.
\end{lma}
\begin{proof}
 First, write
 \[
  Y_{1} = J_{1}\sqcup\cdots\sqcup J_{n_{1}}\sqcup \{q_{1}\}\sqcup\cdots\sqcup\{ q_{m_{1}} \}
 \]
with $J_{k}$ closed non-degenerate intervals for every $k$.

Fix $m\in\mathbb{N}$ such that $1/m<\eta /2$ and define $s_{i}=i/m$ for every $0\leq i\leq m$. Since $g_{1}$ is uniformly continuous, there exists some $\gamma >0$ such that $\diam (g_{1}(A))<1/m$ whenever $\diam (A)<\gamma$.

Set $\epsilon =\gamma /2$ and fix an onto $\epsilon$-map $g_{2}\colon X\rightarrow Y_{2}$ as in the statement of the lemma. Recall that there exists $\delta >0$ such that $\diam (g_{2}^{-1}(Z) )<2\epsilon =\gamma $ whenever $\diam (Z)<\delta$.

Now fix $n\in\mathbb{N}$ such that $1/n<\delta /2$ and define $t_{j}=j/n$ for every $0\leq j\leq n$.

As before, write $Y_{2}$ as
\[
 Y_{2} = I_{1}\sqcup\cdots\sqcup I_{n_{2}}\sqcup \{p_{1}\}\sqcup\cdots\sqcup\{ p_{m_{2}} \}
\]
with $I_{l}$ closed non-degenerate intervals for every $l$.

For each $l,k$, consider the subsets
\[
 \begin{split}
  A_{1}^{k} &= [s_{0},s_{1}) \, ,\, 
  A_{i}^{k} = (s_{i-1},s_{i+1})\, ,\,
  A_{m}^{k} = (s_{m-1},s_{m}] \subset
  J_{k}
  \\
  B_{1}^{l} &=  [t_{0},t_{1}) \, ,\, 
  B_{j}^{l} = (t_{j-1},t_{j+1})\, ,\,
  B_{n}^{l} = (t_{n-1},t_{n}] \subset
  I_{l} .
 \end{split}
\]

By construction, we know that $\diam (B_{j}^{l})< \delta$ for every fixed $j,l$. This implies that $\diam ( g_{2}^{-1}(B_{j}^{l}))< \gamma$ and, consequently,
\[
 g_{1} ( g_{2}^{-1} (B_{j}^{l}) )\subset A_{i}^{k}\,\text{ or }\,
 g_{1} ( g_{2}^{-1} (B_{j}^{l}) )\subset \{ q_{r'}\}
\]
for some $i,k,r'$. Here, we have taken $m$ large enough so that the distance between the connected components of $Y_{1}$ is greater than $1/m$.

Note that, by the same argument, we also have that for every $r\leq m_{2}$ there exist $i,k$ with $g_{1} ( g_{2}^{-1} (p_{r}) )\subset A_{i}^{k}$ or $g_{1} ( g_{2}^{-1} (p_{r}) )\subset \{ q_{r'}\}$ for some $r'$.

Moreover, since for every fixed $l$ and every $j$ we have $B^{l}_{j}\cap B^{l}_{j+1}\neq \emptyset$, we get
\[
 \emptyset \neq g_{1}(g_{2}^{-1}(B^{l}_{j}))\cap g_{1}(g_{2}^{-1}(B^{l}_{j+1}))
\]
because $g_{2}$ is onto.

It follows that, for every fixed $l$, the sets $g_{1}(g_{2}^{-1}(B_{j}^{l}))$ belong to the same connected component of $Y_{1}$ for every $j$. Thus, for every $l$ there either exists $k$ such that $g_{1}(g_{2}^{-1}(I_{l}))\subset J_{k}$ or there exists $r'$ such that $g_{1}(g_{2}^{-1}(I_{l}))=\{q_{r'}\}$.
\vspace{0.1cm}

Given a connected component $Y$ of $Y_{2}$, we define the map $\varphi_{Y}\colon Y\to Y_{1}$ as follows:

 If $g_{1}(g_{2}^{-1}(Y))=\{ q_{r'}\}$ for some $r'$, define $\varphi_{Y}\colon Y\to \{ q_{r'}\}$ as the constant map.
 
 Else, there exists some $k$ such that $g_{1}(g_{2}^{-1}(Y))\subset J_{k}$. If $Y$ is degenerate, we can find $A_{i}^{k}\subset J_{k}$ such that $g_{1}(g_{2}^{-1}(Y))\subset A_{i}^{k}$ for some $i,k$. Define $\varphi_{Y}\colon Y\to A_{i}^{k}\subset J_{k}$ as the constant map $\varphi_{Y}\equiv s_{i}$.
 
 Finally, if $Y$ is non-degenerate, it is of the form $Y=I_{l}$. Then, for every $j$,
 fix $i(j)$ such that $g_{1}(g_{2}^{-1}(B_{j}^{l}))\subset A_{i(j)}^{k}$, and recall that 
 \[
 \emptyset \neq g_{1}(g_{2}^{-1}(B^{l}_{j}))\cap g_{1}(g_{2}^{-1}(B^{l}_{j+1}))
\]
for every $j$. This shows that $\vert i(j)-i(j+1)\vert \leq 1$.

Define $\varphi_{I_{l}}\colon I_{l}\rightarrow J_{k}$ as $\varphi_{I_{l}}(t_{j})=s_{i(j)}$ and extend it linearly.
\vspace{0.1cm}

Let $\varphi:=\varphi_{I_{1}}\sqcup \cdots\sqcup \varphi_{I_{n_{2}}}\sqcup \varphi_{p_{1}}\sqcup\cdots\sqcup \varphi_{p_{m_{2}}}\colon Y_{2}\rightarrow Y_{1}$, which is clearly continuous. We will now see that $\vert g_{1}(x)-\varphi g_{2}(x)\vert <\eta$.

Thus, let $x\in X$ and let $B\subset Y_{2}$ such that $g_{2}(x)\in B$ with $B$ being either $B^{l}_{j}$ for some $l,j$ or $\{ p_{r}\}$ for some $r$. Note that $g_{1}(x)\in g_{1}(g_{2}^{-1}(B))$.

Thus, if $g_{1}(g_{2}^{-1}(B))=\{ q_{r'}\}$ for some $r'$, we have $g_{1}(x)=q_{r'}$ and, consequently,
\[
 \vert g_{1}(x)-\varphi g_{2}(x)\vert =
 \vert q_{r'}-q_{r'}\vert =0
\]
by the definition of $\varphi$.

Finally, if $g_{1}(g_{2}^{-1}(B))\subset A^{k}_{i}$, we have that $g_{1}(x)\in A^{k}_{i}$. Therefore, one gets $\vert g_{1}(x)-s_{i}\vert <1/m$. There are now two different situations:

If $B=\{ p_{r}\}$ for some $r$, we have defined $\varphi_{p_{r}}$ as the constant map $s_{i}$. Thus, one gets
 \[
  \vert g_{1}(x)-\varphi g_{2}(x)\vert =
  \vert g_{1}(x)-s_{i}\vert <1/m<\eta /2.
 \]
 
 Else, if $B=B^{l}_{j}$ for some $l$ and $j$, let $i(j)$ be the previously fixed integer such that $g_{1}(g_{2}^{-1}(B))\subset A^{k}_{i(j)}$.

Then, since $g_{2}(x)\in B$, we either have $t_{j-1}\leq g_{2}(x)\leq t_{j}$ or $t_{j}\leq g_{2}(x)\leq t_{j+1}$. This implies that $\varphi (g_{2}(x))$ is either between  $s_{i(j-1)}$ and $s_{i(j)}$ or between $s_{i(j)}$ and $s_{i(j+1)}$.

Since $\vert i(j)-i(j+1)\vert \leq 1$, the triangle inequality implies that
\[
 \vert g_{1}(x)-\varphi g_{2}(x)\vert \leq 
 \vert g_{1}(x)-s_{i(j)} \vert +
 \vert s_{i(j)}-\varphi g_{2}(x)\vert \leq
 2/m<\eta
\]
as required.
\end{proof}
%===========================================

In order to prove our next result, we will need the following proposition, a proof of which can be found in \cite[Proposition 12.18]{Nadl92}.

\begin{prp}\label{prp:Nad}
 Let $(X,d)$ be a compactum and let $Y= \varprojlim (Y_{i},f_{i})$ be an inverse limit of compacta $(Y_{i},d_{i})$ with $f_i\colon Y_{i+1}\to Y_i$.
 
 Assume that there exist two sequences of strictly positive real numbers $(\delta_{i})$, $(\epsilon_{i})$ with $\lim\epsilon_{i}=0$ and a family of onto $\epsilon_{i}$-maps $g_{i}\colon X\rightarrow Y_{i}$ such that the following conditions hold
 \begin{enumerate}[(i)]
  \item For every pair $i<j$, we have $\diam (f_{i,j}(A))\leq \delta_{i}/2^{j-i}$ for any $A\subset Y_{j}$ with $\diam (A)\leq\delta_{j}$.
  \item $d_{i}( g_{i}(x),g_{i}(y) )>2\delta_{i}$ whenever $d(x,y)\geq 2\epsilon_{i}$.
  \item $d_{i}(g_{i},f_{i}g_{i+1})\leq \delta_{i}/2$.
 \end{enumerate}

 Then, $X\cong Y$.
\end{prp}
%===========================================

We summarize our results in the following.

\begin{thm}\label{achain_iff_comAI}
 Let $X$ be a compactum. The following are equivalent:
 \begin{enumerate}[(i)]
  \item $X$ is almost chainable.
  \item $X$ is a generalized arc-like space.
  \item $X$ is homeomorphic to an inverse limit of finite disjoint copies of unit intervals.
  \item $C(X)$ is an AI-algebra.
 \end{enumerate}
\end{thm}
\begin{proof}
Conditions (i) and (ii) are equivalent by \autoref{eq_AC}, while (iii) is equivalent to (iv) by the arguments in \autoref{pgr:Thom}. Further, that (iii) implies (i) follows from \autoref{necessity}. Thus, we are left to prove that (i) implies (iii).

 Let $X$ be an almost chainable compactum. Then, for every $\epsilon >0$ there exists an onto $\epsilon$-map $f\colon X\rightarrow Y$ with $Y$ a finite disjoint union of closed intervals.
 
 As in \cite[Theorem 12.19]{Nadl92}, we will inductively construct sequences of maps and real numbers satisfying the conditions of \autoref{prp:Nad}. We give the proof here for the sake of completeness, although the only difference with the original proof is that we replace $[0,1]$ with $Y$.\vspace{0.1cm}
 
 Let $0<\epsilon_{1}\leq 1$, and consider an onto $\epsilon_{1}$-map $g_{1}$ from $X$ to a finite disjoint union of closed intervals $Y_{1}$.
 
 Note that, in particular, we must have $g_{1}(x)\neq g_{1}(y)$ whenever $d(x,y)\geq 2 \epsilon_{1}$. By compactness of $X$, this implies that there exists some $\delta_{1} >0$ such that $\vert g_{1}(x)-g_{1}(y)\vert >2\delta_{1} $ whenever $d(x,y)\geq 2\epsilon_{1}$.
 
 Setting $\eta =\delta_{1}/2$, let $\epsilon >0$ be as in our previous lemma and set $\epsilon_{2}=\min \{ 1/2,\epsilon \}$. Then, let $g_{2}\colon X\rightarrow Y_{2}$ be an onto $\epsilon_{2}$-map given by the almost chainability of $X$.
 
 By our choice of $\epsilon_{2}$, there exists some $f_{1}\colon Y_{2}\rightarrow Y_{1}$ such that
 \[
  \vert g_{1}(x)-f_{1}g_{2}(x) \vert <\eta=\delta_{1}/2
 \]
for every $x\in X$.

As before, there exists $\delta_{2}>0$ such that $\vert g_{2}(x)-g_{2}(y)\vert >2\delta_{2} $ whenever $d(x,y)\geq 2\epsilon_{2}$. Furthermore, by the uniform continuity of $f_{1}$, we can choose $\delta_{2}>0$ so that $\diam (f_{1}(A))\leq\eta$ whenever $\diam (A)\leq \delta_{2}$.

This shows that $g_1, g_2$ satisfy conditions (i)-(iii) of \autoref{prp:Nad}.

Proceeding as in \cite[Theorem 12.19]{Nadl92}, one can now inductively find  $\epsilon_{i},\delta_{i}>0$ and maps $g_{i}\colon X\rightarrow Y_{i}$, $f_{i}\colon Y_{i+1}\rightarrow Y_{i}$ satisfying conditions (i)-(iii) \autoref{prp:Nad}.
\vspace{0.1cm}

Applying \autoref{prp:Nad}, we have $X\cong \lim (Y_{k},f_{k})$ with $Y_{k}$ finite disjoint unions of closed intervals. Clearly, this implies that $C(X)$ is an AI-algebra, as required.\vspace{0.1cm}
\end{proof}
%===========================================
%===========================================

\Msection{$\Lsc$-like $\Cu$-semigroups}{Lsc-like Cu-semigroups}\label{Lsc_like_Cu_sem}

In this section, we introduce $\Lsc$-like \CuSgp{s} (\autoref{dfn:LscLikeCuSgp}) and prove some of their main properties. As we shall prove in \autoref{LscXS_iff_Lsc_like}, such \CuSgp{s} are exactly those that are $\Cu$-isomorphic to the \CuSgp{} of lower-semicontinuous functions $\Lsc (X,\NNbar )$ for some $T_{1}$ topological space.

Using \autoref{wayb_decomp}, we also prove that the semigroup $\Lsc (X,\overline{\NN})$ is a \CuSgp{} whenever $X$ is compact and metric; see \autoref{cor:LscXCu}.

\begin{dfn}
 A \CuSgp{} $S$ will be called \emph{distributively lattice ordered} if $S$ is a distributive lattice such that, given any pair of elements $x,y\in S$, we have $x+y=(x\vee y) + (x\wedge y)$.
 
 We will say that a distributively lattice ordered \CuSgp{} is \emph{complete} if suprema of arbitrary sets exist.
\end{dfn}
%===========================================

\begin{rmk}
 Given two increasing sequences $(x_{n})_{n}$ and $(y_{n})_{n}$ in a complete distributively lattice ordered $\Cu$-semigroup, we have that $(\sup_{n} x_{n})\vee (\sup_{n} y_{n})=\sup_{n} (x_{n}\vee y_{n})$.
 
 Indeed, by definition we know that $\sup_{n} x_{n}=\vee_{n=1}^{\infty} x_{n}$ and the equality 
 \[
  (\vee_{n=1}^{\infty} x_{n})\vee (\vee_{n=1}^{\infty} y_{n})=
\vee_{n=1}^{\infty} (x_{n}\vee y_{n})
 \]
holds in any complete lattice.
\end{rmk}

Throughout the paper, we will say that a sum of finitely many indexed elements $x_{1}+\cdots+x_{n}$ is \emph{ordered} if the sequence $(x_{i})_{i=1}^{n}$ is increasing or decreasing.
%===========================================

\begin{lma}\label{std_formula}
 Let $S$ be a distributively lattice ordered $\Cu$-semigroup. Given two finite decreasing sequences $(x_{i})_{i=1}^{m}, (y_{i})_{i=1}^{m}$, the following equality holds
 \[
  \sum_{i=1}^{m} (x_{i}+y_{i})=\sum_{i=1}^{2m} \vee_{j=0}^{m}(x_{j}\wedge y_{i-j})
 \]
where, in the right hand side, $x_{i}\wedge y_{k}=x_{i}$, and $x_{k}\wedge y_{i} =y_{i}$ whenever $k\leq 0$ and $y_{k}=x_{k}=0$ whenever $k>m$.

Note that the sum in the right hand side is decreasingly ordered.
\end{lma}
\begin{proof}
 This is a generalization of the equality $x+y=(x\vee y) + (x\wedge y)$ and is proven by induction.
\end{proof}

\begin{rmk}\label{OFS}
 Given any finite sum $x_{1}+\cdots +x_{n}$, we can apply \autoref{std_formula} iteratively (first to $x_{1}+x_{2}$, then to $((x_{1}\vee x_{2}) +(x_{1}\wedge x_{2}))+x_{3}$, etc.) to obtain an ordered sum.
 
 Therefore, every sum in a distributively lattice ordered semigroup can be written as an ordered sum.
\end{rmk}
%===========================================

\begin{dfn}
 Let $S$ be a \CuSgp{} and let $H$ be a subset of $S$. We say that $H$ is \emph{topological} if, given two finite increasing sequences $(x_{i})_{i=1}^{m}, (y_{i})_{i=1}^{m}$ in $H$, we have 
 \begin{equation*}
  \sum_{i=1}^{m} x_{i}\leq \sum_{i=1}^{m} y_{i}
 \end{equation*}
 if and only if $x_{i}\leq y_{i}$ for every $i$.
\end{dfn}
%===========================================

Given any element $r$ in partially ordered set $P$,   we denote by $\downarrow r$ the set $\{s\in P\mid s\leq r\}$; see, for example, \cite[Definition~O-1.3]{GieHof+03Domains}.
%===========================================

\begin{ntn}
 Given a \CuSgp{} $S$ and an element $y\in S$, we write $\infty y:=\sup_n ny$. Further, if $S$ has a greatest element, we denote it by $\infty$.
\end{ntn}

\begin{dfn}\label{dfn:LscLikeCuSgp}
 A \CuSgp{} $S$ will be said to be \emph{$Lsc$-like} if it is a complete distributively lattice ordered \CuSgp{} such that the following conditions hold:
 \begin{enumerate}[(C1)]
  \item For every pair of idempotent elements $y,z$ in $S$, $y\geq z$ if and only if
  \[
   \{ x<\infty \mid x \text{ maximal idempotent},\, 
   x\geq y
   \}
   \subseteq 
   \{ x<\infty \mid x \text{ maximal idempotent},\, 
   x\geq z
   \}.
  \]
  \item There exists a topological subset of the form $\downarrow e$ such that the finite sums of elements in $\downarrow e$ are sup-dense in $S$.
\end{enumerate}
\end{dfn}
%===========================================

The following example justifies our terminology.
 
 \begin{exa}\label{exa_Lsc_is_Lsc}
  Any \CuSgp{} of the form  $\Lsc (X,\NNbar )$ with $X$ a $T_{1}$ space is $\Lsc$-like. Indeed, $\Lsc (X,\NNbar )$ is clearly a complete distributively lattice ordered semigroup.
  
  Further, the maximal idempotent elements $s<\infty$ are of the form $s=\infty \chi_{X\setminus\{ x \}}$, and given any pair of elements $\infty f=\infty \chi_{\text{supp}(f)}$ and $\infty g=\infty \chi_{\text{supp}(g)}$, we know that $\infty \chi_{\text{supp}(f)}\leq \infty\chi_{\text{supp}(g)}$ if and only if $\text{supp}(f)\subset \text{supp}(g)$. That is, if and only if for every $x\in X$ such that $\text{supp}(g)\subset X\setminus\{ x\}$ we have $\text{supp}(f)\subset X\setminus\{ x\}$. This shows that $\Lsc (X,\NNbar )$ satisfies (C1).
  
  To see (C2), consider the subset $\downarrow 1$. One can see that the order is topological, and it clearly generates a semigroup that is dense in $\Lsc (X,\NNbar )$.
 \end{exa}
%===========================================

Recall that by an \emph{order unit} we mean an element $e$ such that $x\leq \infty e$ for every $x\in S$.

\begin{rmk}
 Let $S$ be an $\Lsc$-like \CuSgp{}. Since the semigroup generated by $\downarrow e$ is dense in $S$, $e$ is an order unit.
 
 Furthermore, let $f$ be another order unit, and let $(f_{m})_{m}$ be an increasing sequence of ordered finite sums of elements below $e$ such that $f=\sup_{m} f_{m}$ (the sums can be assumed to be ordered by \autoref{OFS}). Write $f_{m}=\sum_{i=1}^{r_{m}}g_{i,m}$ with $e\geq g_{i,m}\geq g_{i+1,m}$ for each $i< r_{m}$.

Also consider an $\ll$-increasing $(e_{n})_{n}$ whose supremum is $e$.

Then, since $e\leq \infty f$, one has that for every $n\in\mathbb{N}$ there exist some $k,m\in\mathbb{N}$ with $e_{n}\leq kf_{m}=\sum_{i=1}^{r_{m}} kg_{i,m}$. Thus, since the order in $\downarrow e$ is topological and $\sum_{i=1}^{r_{m}} kg_{i,m}$ is an ordered sum of elements below $e$ (the first $k$ greatest elements are $g_{1,m}$, the next $k$ are $g_{2,m}$, etc.), it follows that $e_{n}\leq g_{1,m}\leq f$.

This shows $e\leq f$  and, consequently, $e$ is the least order unit of $S$.
\end{rmk}
%===========================================

\begin{pgr}
 Note that an $\Lsc$-like \CuSgp{} $S$ has no maximal idempotent elements $x<\infty$ if and only if $S= \{ 0\}$. Indeed, if $S$ has no maximal idempotents then every idempotent $z$ satisfies $z\leq 0$ by condition (C1). This implies that every idempotent element is zero and, since for every $s\in S$ the element $\infty s$ is idempotent and $\infty s\geq s$, we get $s=0$ for every $s$, as desired. The converse is trivial.
 
 Similarly, an element $s\in S$ satisfies $\infty s=\infty$ if and only if there are no maximal idempotents $x<\infty$ with $\infty s\leq x$.
\end{pgr}
%===========================================

\begin{lma}\label{lma:decomp1}
 Let $S$ be an $\Lsc$-like \CuSgp{} and let $y\leq ne$, where $e$ is the least order unit of $S$ and $n\in\NN$. Then, $y$ can be written as an ordered sum of at most $n$ non-zero terms in $\downarrow e$.
\end{lma}
\begin{proof}
 Write $y$ as $y=\sup_{r} y_{r}$, where each $y_{r}$ is a finite ordered sum of elements in $\downarrow e$; see \autoref{std_formula} and \autoref{OFS}.
 
 As $y_{r}\leq ne$ for each $r$ and $\downarrow e$ is topological, each $y_{r}$ has at most $n$ non-zero summands.
 
 Further, since $y_{r}\leq y_{r+1}$, the $k$-th summands of $(y_{r})_{r}$ form an increasing sequence for every $k$.
 
 Taking their suprema, this shows that $y$ is an ordered sum of at most $n$ elements in $\downarrow e$.
\end{proof}
%===========================================

\begin{cor}\label{Sum_cup}
 Let $S$ be an $\Lsc$-like \CuSgp{}. Given $y_{1},\cdots ,y_{n}\in S$ such that $y_{1}+\cdots +y_{n}\geq s$ for some $s\leq e$, we have $y_{1}\vee \cdots \vee y_{n}\geq s$.
\end{cor}
\begin{proof}
 Take any $s'\ll s$ and let $y'_{i}\ll y_{i}$ be such that $y'_{1}+\cdots +y'_{n}\geq s'$. Since $\infty e\geq y_{1}+\cdots + y_{n}\gg y'_{1}+\cdots +y'_{n}$, there exists some $k\in\NN$ such that $ke\geq y'_{1}+\cdots +y'_{n}$.

By \autoref{lma:decomp1} above, each $y'_{i}$ can be written as an ordered sum of at most $k$ elements below $e$. That is, we can write $y'_{i}=\sum_{j=1}^{k} y_{j,i}$ with $e\geq y_{j,i}\geq y_{j+1,i}$ for each $j<k$.

Applying \autoref{std_formula} iteratively, we have that $y'_{1}+\cdots +y'_{n}$ can be written as a finite ordered sum of elements below $e$, and that the greatest summand of the sum is $y_{1,1}\vee\cdots \vee y_{1,n}$.

Since the order in $\downarrow e$ is topological, we get $y_{1,1}\vee\cdots \vee y_{1,n}\geq s'$. It follows that 
\[
 s'\leq y_{1,1}\vee\cdots \vee y_{1,n}\leq y'_{1}\vee\cdots\vee y'_{n}\leq y_{1}\vee\cdots \vee y_{n}.
\]

Since this holds for every $s'\ll s$, we have $s\leq y_{1}\vee\cdots \vee y_{n}$ as required.
\end{proof}
%===========================================

The next two lemmas will be of particular importance when working with the induced topology of an $\Lsc$-like \CuSgp{}; see \autoref{pgr:TopSpaCuSgp}. In particular, \autoref{Comparatibility} below gives an alternative version of (C1) in \autoref{dfn:LscLikeCuSgp}.

\begin{lma}\label{lma:IntInf}
 Let $S$ be an $\Lsc$-like \CuSgp{} with least order unit $e$, and let $x\in S$. Then, $\infty x=\infty (x\wedge e)$ and $(\infty x)\wedge e=x\wedge e$.
 
 In particular, if $x<\infty$ is a maximal idempotent, we  must have $x\wedge e\neq e$.
\end{lma}
\begin{proof}
 By \autoref{lma:decomp1} and condition (C2) in \autoref{dfn:LscLikeCuSgp}, $x$ can be written as the supremum of finite ordered sums of elements in $\downarrow e$. Moreover, using that the order in $\downarrow e$ is topological, it follows that the sequence formed by the greatest element of each ordered sum is increasing. Let $x'$ be the supremum of such a sequence.
 
 Then, it is clear that $\infty x=\infty x'$ and $x\wedge e=x'$. Therefore, we have 
 \[
  \infty x=\infty x' = \infty (x\wedge e),
  \andSep
  (\infty x)\wedge e = (\infty x')\wedge e = x'=x\wedge e,
 \]
 as desired.
\end{proof}
%===========================================

\begin{lma}\label{Comparatibility}
 Let $S$ be an $\Lsc$-like \CuSgp{} with least order unit $e$. Let $y,z\leq e$. If $z\leq x$ for every maximal element $x<e$ such that $y\leq x$, then $z\leq y$.
\end{lma}
\begin{proof}
If there are no maximal idempotent elements $x<\infty$ we know that $S=\{ 0\}$, so we may assume otherwise.

We claim that the maximal idempotent elements $x<\infty$ are precisely the elements $\infty s$ with $s<e$ maximal.

To see this, let $x<\infty$ be a maximal idempotent and take $s=x\wedge e$. By  \autoref{lma:IntInf} we have $\infty s= \infty (x\wedge e)=\infty x=x$. Let us check that $s$ is maximal, and thus take $t\in S$ with $s\leq t\leq e$.

By maximality of $x$, we either have $\infty t=\infty $ or $\infty t=x$. Using \autoref{lma:IntInf} once again, it follows that we have $t= (\infty t)\wedge e=\infty\wedge e = e$ or $t=x\wedge e=s$, as required.

 Conversely, let $s<e$ be maximal and consider the element $\infty s=\sup_{n} ns$. Let $(s_k)_k$ be an $\ll$-increasing sequence with supremum $s$.
 
 Given an idempotent element $x$ such that $\infty s\leq x\leq \infty$, we know by \autoref{lma:decomp1} that there exists an $\ll$-increasing sequence $(x_{m})_m$ with supremum $x$ such that each $x_m$ can be written as a finite increasing sum of elements in $\downarrow e$. 
 
 Let $(m_k)_k$ be an increasing sequence of integers such that $ks_{k}\leq x_{m_{k}}$ for every $k$.
 
 Using that $\downarrow e$ is topological, each $s_{k}$ is less than or equal to each of the first $k$ summands of $x_{m_{k}}$. As in the proof of \autoref{lma:IntInf}, note that the greatest summands of each sum form an increasing sequence of elements below $e$. Letting $x'$ be the supremum of this sequence, we get $s\leq x'\leq e$.
 
 By maximality of $s$, we either have $s=x'$ or $x'=e$. Thus, it follows from \autoref{lma:IntInf} that $\infty s= x$ or $x=\infty x'=\infty$. This proves that  $\infty s$ is a maximal idempotent, as desired.
 
 A similar argument shows that for any $y,z\leq e$ we have $y\leq z\leq e$ if and only if $\infty y\leq \infty z$. Consequently, if $z\leq s$ for every maximal element $s<e$ such that $y\leq s$, we know that $\infty z\leq \infty y$. Taking the infimum with $e$, one gets $z=(\infty z)\wedge e\leq (\infty y)\wedge e=y$ as required.
\end{proof}

\begin{rmk}
 Note that the previous lemma implies that $y<e$ if and only if there exists a maximal element $x<e$ with $y\leq x$.
\end{rmk}
%===========================================

Recall that a complete lattice $(P,\leq )$ is said to be a \emph{complete Heyting algebra} if for every element $s\in P$ and every subset $T\subseteq P$, the following holds:
\[
 s\wedge (\vee_{t\in T} t)=\vee_{t\in T}(s\wedge t).
\]

\begin{lma}\label{prp:Heyting}
 Let $S$ be an $\Lsc$-like \CuSgp{} with least order unit $e$. The subset $\downarrow e\subset S$ is a complete Heyting algebra.
\end{lma}
\begin{proof}
 By definition, we have to see that for every subset $T\subset \downarrow e$, one has $s\wedge (\vee_{t\in T} t)=\vee_{t\in T}(s\wedge t)$ for every $s\in \downarrow e$. Thus, let $x<e$ be maximal with $x\geq \vee_{t\in T}(s\wedge t)$, which happens if and only if $x\geq s\wedge t$ for every $t\in T$. Since $x\neq e$ and $x\vee x\geq x\vee (s\wedge t)=(x\vee s)\wedge (x\vee t)$, we either have $x\vee s=x$ or $x\vee t=x$ (since otherwise both of these unions would be equal to $e$ and then $x\geq e$, a contradiction).

If $x= x\vee s \geq s$, we have $x\geq s\wedge (\vee_{t\in T} t)$. Else, $x\geq t$ for each $t\in T$, so we also get $x\geq s\wedge (\vee_{t\in T} t)$.

By \autoref{Comparatibility}, this shows that $s\wedge (\vee_{t\in T} t)\leq \vee_{t\in T}(s\wedge t)$. 

The other inequality holds in any lattice.
\end{proof}
%===========================================

Given a \CuSgp{} $S$ where every pair of elements has a supremum, we say that $S$ is \emph{sup-semilattice ordered} if for every $x,y,z\in S$ we have $x+(y\vee z)=(x+y)\vee (x+z)$.

A natural question to ask about $\Lsc$-like \CuSgp{s} is if they satisfy such a property. This is indeed the case:

\begin{lma}\label{lma:supsemi}
 Let $S$ be an $\Lsc$-like \CuSgp{} $S$ and let $e$ be its least order unit. Given $y\leq z$ with $y\leq ne$ for some $n$, there exists an element $y\setminus z\in S$ such that, for $x\in S$, we have $x+y\leq z$ if and only if $x\leq y\setminus z$. 
 
 In particular, given $x,y,z\in S$, we have $x+ (y\vee z)=(x+y)\vee (x+z)$.
\end{lma}
\begin{proof}
We will construct our \textquoteleft almost-complement\textquoteright{ } $y\setminus z$ in three steps:
\vspace{0.1cm}

\noindent \emph{Step 1.} Let $e$ be the least order unit of $S$ and assume that $y\leq z\leq e$. Then, consider the subset $T:=\{ x\in S\mid y+x\leq z \}$ and, since arbitrary suprema exist in $S$, we define
 \[
  y\setminus z :=\vee \{ x\in S\mid y+x\leq z \} = \vee_{x\in T}x.
 \]
 
Further, note that, for $x\in T$, $y+x=(y\vee x)+(y\wedge x)$ and, as $z\leq e$ and $\downarrow e$ is topological, we have $y\vee x\leq z$ and $y\wedge x=0$.

Using that $\downarrow e$ is a complete Heyting algebra, we get $(\vee_{x\in T}x)\vee y=\vee_{x\in T} (x\vee y)\leq z$ and $(\vee_{x\in T}x)\wedge y=\vee_{x\in T}(x\wedge y)=0$. Consequently, we get
\[
 y+(y\setminus z)=y\vee (y\setminus z)+y\wedge (y\setminus z)=y\vee (y\setminus z)\leq z.
\]

This shows that $x\leq y\setminus z$ if and only if $y+x\leq z$.
\vspace{0.1cm}

\noindent \emph{Step 2.} Now, given any pair of elements $y\leq z\leq ne$ for some $n\in\mathbb{N}$, write them as finite ordered sums of elements in $\downarrow e$, say $y=\sum_{i=1}^{n} y_{i}$ and $z=\sum_{j=1}^{n} z_{j}$.

Since $\downarrow e$ is topological, $y_{i}\leq z_{i}\leq e$ for every $i$, so we can define $y\setminus z:=\sum_{i=1}^{n} y_{i}\setminus z_{i}$.

Note that, given any element $x$ such that $y+x\leq z\leq ne$, we have $x\leq ne$ and so, by \autoref{lma:decomp1}, $x$ can be written as a finite ordered sum $\sum_{i=1}^{n} x_{i}$ with $x_i\leq e$ for every $i$. Moreover, we also have
\[
 \sum_{i=1}^{n} y_{i}\vee x_{i}\leq y+x\leq \sum_{i=1}^{n} z_{i}
\]
and so $y_{i}\vee x_{i}\leq z_{i}$, which happens if and only if $x_{i}\leq y_{i}\setminus z_{i}$.

This shows that $x+y\leq z\leq ne$ for some $n$ if and only if $x\leq y\setminus z$.
\vspace{0.1cm}

\noindent \emph{Step 3.} Given $y\leq z\wedge ne$ for some $n$, we have that $y\leq z\wedge me$ for every $m\geq n$.

Thus, we can consider the element $y\setminus (z\wedge me)$. Further, it is easy to check that $y\setminus (z\wedge me)\leq y\setminus (z\wedge (m+1)e)$ for every $m$. We define $y\setminus z:=\sup_{m} y\setminus(z\wedge me)$, which has the required property.
\vspace{0.1cm}

Finally, to see that $x+ (y\vee z)=(x+y)\vee (x+z)$ for any given $x,y,z$, note that \textquoteleft{}$\geq $\textquoteright{ } is clear. To prove \textquoteleft{}$\leq $\textquoteright, let $x'\ll x$ and let $s$ be such that $x'+(y\vee z)\leq s$. 

Since $x'\ll x\leq \infty e$, there exists some $n$ with $x'\leq ne$, so we can consider the element $x'\setminus s$.

Thus, we know that $x'+(y\vee z)\leq s$ holds if and only if $y\vee z\leq x'\setminus s$, which in turn holds if and only if $x'+y,x'+z\leq s$. Consequently, we get
\[
 x'+(y\vee z)=(x'+y) \vee (x'+z).
\]

Since the equality holds for every $x'\ll x$, it also holds for $x$.
\end{proof}
%===========================================

\begin{lma}\label{pgr:canc}
Let $S$ be an $\Lsc$-like \CuSgp{} with least order unit $e$, and let $x,y,z\leq e$ with $x+y\leq x+z$. Then, $y\leq z$.
\end{lma}
\begin{proof}
First, note that $y\leq x\vee y\leq x\vee z$. Indeed, since $x,y,z\leq e$ and we know that $x+y=(x\vee y)+(x\wedge y)$ and $x+z=(x\vee z) +(x\wedge z)$, we have that $(x\vee y)+(x\wedge y)\leq (x\vee z) +(x\wedge z)$.

Since the right and left hand side of the previous inequality are ordered sums of elements below $e$, we can use that the order in $\downarrow e$ is topological to get $x\vee y\leq x\vee z$, as desired.

Therefore, the sum $(z\vee x)+y$ is ordered, as we have $y\leq z\vee x$.

Using \autoref{lma:supsemi} at the first and third step, the inequality $x+y\leq x+z$ at the second step, and that $S$ is distributively lattice ordered at the last step, one obtains
\[
\begin{split}
 (z\vee x) +y &= (y+z)\vee (y+x) \leq (y+z)\vee (z+x) = z+(y\vee x)\\
 &=(z\vee y\vee x) + z\wedge (y\vee x).
 \end{split}
\]

Using once again that the order in $\downarrow e$ is topological, it follows that $y\leq z\wedge (y\vee x)\leq z$, as required.
\end{proof}

By extending the previous proof, one can check that whenever $y,z\in S$ and $x\leq ne$ for some $n$, the same cancellation property holds. In particular, it follows that every $\Lsc$-like \CuSgp{} has weak cancellation.

\begin{rmk}
 Using the previous form of cancellation and the equality
 \[
 \begin{split}
  ((x+y)\vee (x+z)) + (x+(y\wedge z))&=x+y+x+z\\
  &=((x+y)\vee (x+z)) + ((x+y)\wedge (x+z)),
  \end{split}
 \]
one can check that $S$ is inf-semilattice ordered, that is to say, $x+y\wedge z=(x+y)\wedge (x+z)$ for every $x,y,z$; see \cite{AntPerRobThi21:CuntzSR1}. For elements $x,y,z\leq ne$ for some $n$, note that this simply follows by cancelling the term $((x+y)\vee (x+z))$.

Using that every element in $S$ is the supremum of an $\ll$-increasing sequence of finite sums of elements below $e$, one can then prove that the equality $x+(y\wedge z)=(x+y)\wedge (x+z)$ is always satisfied.
\end{rmk}
%===========================================

As one might expect, having a topological order also affects the way below relation. 

\begin{lma} \label{wayb_C5}
 Let $S$ be an $\Lsc$-like \CuSgp{} with least order unit $e$. If $y,z,y',z'\leq e$ are such that
 \[
  y+z\ll y' +z'
 \]
 with $y\geq z$ and $y'\geq z'$, we have that $y\ll y'$ and $z\ll z'$.
 
 The same holds for any pair of finite sums (i.e. with more than two summands)
\end{lma}
\begin{proof}
 Write $y'=\sup_{n} y'_{n}$ and $z'=\sup_{n} z'_{n}$ with $(y'_{n})_{n}$ and $(z'_{n})_{n}$  $\ll$-increasing. Since $y+z\ll y' +z'$, we have $y+z\leq y'_{n}+z'_{n}$ for some $n$.
 
 Moreover, as we have that $z'_{n}\ll z'\leq y'$, there exists a $k$ such that $z'_{n}\leq y'_{k}$. This implies that
 \[
  y+z\leq  y'_{n}+z'_{n}\leq  y'_{\max\{n,k\} }+z'_{n}
 \]
and, since the order is topological, we obtain $y\leq y'_{\max\{n,k\} }\ll y'$ and $z\leq z'_{n}\ll z'$.
\end{proof}
%===========================================

Even though the following lemma is probably well known, we prove it here for the sake of completeness. For second countable finite dimensional compact Hausdorff spaces, it follows from a much more general result that $\Lsc (X,\NNbar )\in\Cu$ (see \cite[Theorem~5.15]{AntPerSan11PullbacksCu}). As a consequence of \autoref{wayb_decomp}, we will have that $\Lsc (X,\NNbar )$ is a \CuSgp{} whenever $X$ is a compact metric space; see \autoref{cor:LscXCu}.

Given $f\in \Lsc (X,\NNbar )$ and $n\in\NN$, we write $\{ f\geq n\}$ to denote the open set $f^{-1}([n,\infty])$. For an open set $U\subseteq X$, we denote by $\chi_U$ the indicator function of $U$.

\begin{lma} \label{wayb_decomp}
 For any topological space $X$ and any pair $f,g\in\Lsc (X,\NNbar )$, one has $f\ll g$ if and only if
 \[
  \chi_{\{ f\geq n\}}\ll \chi_{\{ g\geq n\}} \text{ for every $n$ and } \sup (f)<\infty.
 \]
\end{lma}
\begin{proof} 
 First, let us assume that $f\ll g$. Fix $n\in\mathbb{N}$ and consider an increasing sequence $(h_{k})_{k}$ such that 
 \[
  \chi_{\{ g\geq n\}}\leq \sup_{k} h_{k},
 \]
 which happens if and only if $\chi_{\{ g\geq n\}}\leq\chi_{\cup_{k}\text{supp}(h_{k})}$.
 
 Define the increasing sequence of functions
 \[
  G_{k}:= (n-1) + \chi_{\text{supp}(h_{k})}
  \sum_{r=0}^{\infty }\chi_{\{ g\geq n+r\}}
 \]
and note that $g\leq \sup_{k}G_{k}$.

Since $f\ll g$, we get that $f\leq G_{k}$ for some $k$ and, consequently,
\[
 \{f\geq n\} \subseteq \{G_{k}\geq n\}=\text{supp}(h_{k})\cap \{g\geq n\}\subseteq \text{supp}(h_{k}).
\]

This in turn implies $\chi_{\{f\geq n\}}\leq \chi_{\text{supp}(h_{k})}\leq h_{k}$, so it follows that $\chi_{\{f\geq n\}}\ll \chi_{\{g\geq n\}}$. That $\sup(f)<\infty$ is clear.
\vspace{0.1cm}

Conversely, if $\sup(f)<\infty$, we know that $f=\sum_{i=1}^{m} \chi_{\{ f\geq i \}}$ for some $m<\infty$. Furthermore, given an increasing sequence $(h_{k})_{k}$ with $g\leq \sup_{k}h_{k}$, it follows that
\[
 \{g\geq n\}\subseteq \bigcup_{k} \{h_{k}\geq n\}
\]
so $\chi_{\{ g\geq n\}}\leq \sup_{k}\chi_{\{h_{k}\geq n\}}$. Since $\chi_{\{ f\geq n\}}\ll\chi_{\{ g\geq n\}}$ for every $n$, we get that for each $i$ there exists an integer $k_{i}$ with
\[
 \chi_{\{ f\geq i\}}\leq\chi_{\{ h_{k_{i}}\geq i\}}.
\]

Taking $k=\sup_{i=1}^{m}\{k_{i}\}$, we have
\[
 f\leq \sum_{i=1}^{m} \chi_{\{ h_{k}\geq i\}}\leq h_{k}
\]
as desired.
\end{proof}

\begin{cor}\label{cor:LscXCu}
 Let $X$ be a compact metric space. Then $\Lsc (X,\NNbar )$ is a \CuSgp{} with pointwise order and addition.
\end{cor}
\begin{proof}
 Axioms \axiomO{1} and \axiomO{4} are always satisfied in $\Lsc (X,\NNbar )$, so we only need to prove \axiomO{2} and \axiomO{3}.
 
 First, note that $\chi_U\ll \chi_V$ if and only if $U$ is compactly contained in $V$. Indeed, if $\chi_U\ll \chi_V$, we can write $V$ as a countable increasing union of open sets $V_n$ such that $V_n$ is compactly contained in $V_{n+1}$ for every $n$. Thus, one gets $\chi_U\ll \sup_n \chi_{V_n}$, which implies that $U$ is  contained in $V_n$ for some $n$. Conversely, if $\overline{U}\subset V$ and $(W_n)_n$ is an increasing sequence of open sets with $V=\cup_n W_n$, it is clear that $U\subset W_n$ for some $n$. This shows $\chi_U\ll \chi_V$, as required.
 
 In particular, it follows that that every indicator function can be written as the supremum of a $\ll$-increasing sequence. Since every element in $S=\Lsc (X,\NNbar )$ is the supremum of finite sums of indicator functions, one can check that $S$ satisfies \axiomO{2}.
 
 Now let $f'\ll f$ and $g'\ll g$ in $S$, which by \autoref{wayb_decomp} implies that $\sup (f),\sup (g)<m\leq \infty$ and $\{ f'\geq n \}, \{ g'\geq n\}$ are compactly contained in $\{ f\geq n \}, \{ g\geq n\}$ respectively. Thus, we have
 \[
 \bigcup_{k=0}^{m}\overline{(\{ f'\geq k \}\cap\{ g'\geq n-k \})} \subset \bigcup_{k=0}^{m}(\{ f\geq k \}\cap\{ g\geq n-k \})
 \]
for every $n\leq \sup (f)+\sup (g)$, where note that the left hand side is equal to $\overline{\{ f'+g'\geq n\}}$ and the right hand side is contained in $\{f+g\geq n\}$. By \autoref{wayb_decomp} we have $f'+g'\ll f+g$, which shows that $S$ satisfies \axiomO{3}.
\end{proof}

%===========================================

\Msection{The topological space of an $\Lsc$-like $\Cu$-semigroup}{The topological space of a Lsc-like Cu-semigroup}\label{Topological_Lsc_like}

In this section, we associate to each $\Lsc$-like \CuSgp{} $S$ a topological space $X_S$. In \autoref{prp:BasicTop}, we prove some of the properties that such a topological space must satisfy and, using these, we show that $\Lsc (X_S,\NNbar )$ is always a \CuSgp{}; see \autoref{Lsc_Cu}. In \autoref{LscXS_iff_Lsc_like}, we will see that $S$ and $\Lsc (X_S,\NNbar )$ are in fact $\Cu$-isomorphic.

We also introduce notions for  \CuSgp{s} that have a topological equivalent when the semigroup is $\Lsc$-like.
More explicitly, given an $\Lsc$-like \CuSgp{} $S$, we characterize when $X_{S}$ is second countable, normal and metric in terms of algebraic properties of $S$; see \autoref{prp:TopProp}.

\begin{pgr}\label{pgr:TopSpaCuSgp}
Let $S$ be an $\Lsc$-like \CuSgp{} with least order unit $e$. The \emph{topological space} $X_S$ of $S$ is defined as 
\[
    X_{S}:=\{ x\in S\mid x<e \text{  maximal} \},           
\]
with closed subsets
\[
 C_{y}:=\{ x\in X_{S}\mid x\geq y\}, \quad y\leq e.
\]
\end{pgr}
%===========================================

We check that this is indeed a topology for $X_{S}$.

\begin{lma} \label{topology}
 Let $S$ be an $\Lsc$-like \CuSgp{}. Then, $\{X_S\setminus C_y \mid y\leq e\}$ is a $T_1$-topology for $X_{S}$.
\end{lma}
\begin{proof}
 First note that $C_{0}=X_S$ and that $C_{e}=\emptyset$. Moreover, $C_{x}=\{ x\}$ for every $x\in X_S$. Thus, our topology is $T_1$.
 
 To see that arbitrary intersections of $C_{y}$'s are of the form $C_{z}$ for some $z\leq e$, simply note that
 \[
  \bigcap_{i} C_{y_{i}} = C_{\vee_{i}y_{i}} .
 \]
 Further, one also has that
 \[
  \bigcup_{i=1}^{n} C_{y_{i}} = C_{\wedge_{i=1}^{n}y_{i}} .
 \]
 Indeed, given $x\in X_S$ with $x\geq \wedge_{i=1}^{n}y_{i}$, we have
 \[
  x= x\vee x \geq x\vee (\wedge_{i=1}^{n}y_{i})=(x\vee y_{1})\wedge \cdots\wedge (x\vee y_{n}).
 \]
 
Since $x$ is maximal, for each $i$ we either have $x\vee y_{i}=x$ or $x\vee y_{i}=e$.

However, note that the previous inequality implies that we cannot have $x\vee y_{i}=e$ for every $i$, so $x=x\vee y_{j}\geq y_{j}$ for some $j$.
 
 The other inclusion is clear.
\end{proof}
%===========================================

Retaining the above notation, for every $y\leq e$ we will denote by $U_y$ its \emph{associated open subset}. That is, $U_y=X\setminus C_y$.

We list some properties of these sets. Recall that for every pair of elements $y\leq z\leq e$ the element $y\setminus z$ denotes the almost complement of $z$ by $y$, as constructed in \autoref{lma:supsemi}.

\begin{prp}\label{prp:BasicTop}
 Let $S$ be an $\Lsc$-like \CuSgp{} with least order unit $e$. Then:
 \begin{enumerate}[(i)]
  \item For every $y,z\leq e$, $C_{y}\subseteq C_{z}$ if and only if $y\geq z$.
  \item For every $y\leq e$, $U_y = \{x\in X_{S}\mid y\vee x=e\}$.
  \item Given $y,z\leq e$ such that $U_{y}\subseteq C_{z}$, we have $y\wedge z=0$.
  \item The closure of $U_{y}$, denoted by $\overline{U_{y}}$, is $C_{y\setminus e}$ for every $y\leq e$.
  \item Given $y\leq e$, we have $\Int (C_{y})=X\setminus\overline{(X\setminus C_{y})}=U_{y\setminus e}$.
  \item For every $y,z\leq e$, $C_{y}\subseteq U_{z}$ if and only if $y\vee z=e$.
 \end{enumerate}
\end{prp}
\begin{proof}
 To see (i) recall that, by definition, $C_{y}\subseteq C_{z}$ if and only if $x\geq z$ for every $x\geq y$ with $x<e$ maximal. Using \autoref{Comparatibility}, we see that this is equivalent to $y\geq z$.\vspace{0.1cm}
 
 For (ii), let $y\leq e$ and take  $x<e$ be maximal. Thus, we have $x\leq y\vee x\leq e$. Since $x$ is maximal, we either have $x=y\vee x\geq y$ (i.e. $x\in C_{y}$) or $y\vee x=e$. Thus, $U_{y}= \{ x\in X_S: y\vee x=e \}$.\vspace{0.1cm}
 
 To prove (iii), let us assume, for the sake of contradiction, that $y\wedge z\neq 0$. Then, $U_{y\wedge z}$ is nonempty and we can consider a maximal element $x\in U_{y\wedge z}$.
 
 By (ii), we have $x\vee (y\wedge z)=e$. Thus, one gets $x\vee y=e$ and, consequently, $x\geq z$ from our assumption that $U_y \subseteq C_z$. However, we also have
 \[
  e= x\vee (y\wedge z) = (x\vee y)\wedge (x\vee z) = e\wedge x=x
 \]
 which is a contradiction, as required.\vspace{0.1cm}
 
 Let us now prove (iv) and, consequently, (v). First note that, if $y\vee x=e$, we have
 \[
  y+x=(y\vee x) + (y\wedge x)\geq e\geq y\setminus e +y.
 \]
 
Cancelling $y$ (see \autoref{pgr:canc}), we have that $x\geq y\setminus e$. This shows $U_{y}\subset C_{y\setminus e}$.

Conversely, let $z$ be such that $U_{y}\subset C_{z}$. By (iii), we know that this implies $y\wedge z=0$ and, consequently, $y+z\leq e$. Thus, $z\leq y\setminus e$ or, equivalently, $C_{y\setminus e}\subset C_{z}$.\vspace{0.1cm}

Finally, for (vi), assume first that $C_{y}\subset U_{z}$. Also, assume for the sake of contradiction that $y\vee z\neq e$. Then, there exists $x\in X_{S}$ with $x\geq y\vee z$. This implies $x\geq y$ and, consequently, $x\vee z=e$ from $C_y\subseteq U_z$. 
 However, we have
 \[
  x = x\vee y\vee z \geq x\vee z=e
 \]
which is a contradiction.

Conversely, if $y\vee z=e$, take $x\geq y$, which implies $x\vee z\geq y\vee z=e$. In particular, $x\in U_{z}$.
\end{proof}
%===========================================

\begin{exa}\label{exa:LscSpace}
 Let $X$ be a $T_{1}$ topological space. Recall from \autoref{exa_Lsc_is_Lsc} that $S=\Lsc (X, \NNbar )$ is an $\Lsc$-like \CuSgp{} with least order unit $1$. Then, the topological space of $S$ is homeomorphic to $X$.
 
 Indeed, note that the maximal elements below $1$ are the characteristic functions $\chi_{X\setminus\{ x \}}$. Thus, we have
 \[
  X_{S} =\{ \chi_{X\setminus\{ x \}}\mid x\in X \}
 \]
and 
 \[
  \begin{split}
   U_{\chi_{\mathcal{U}}} &= X_{S}\setminus C_{\chi_{\mathcal{U}}}=
   X_{S}\setminus\{ \chi_{X\setminus\{ x \}} \mid \chi_{X\setminus\{ x \}}\geq \chi_{\mathcal{U}} \}\\
   &= X_{S}\setminus\{ \chi_{X\setminus\{ x \}}\mid x\in X\setminus\mathcal{U} \} = \{ \chi_{X\setminus\{ x \}}\mid x\in\mathcal{U} \}
  \end{split}
 \]
for every open subset $\mathcal{U}\subset X$.

It should now be clear that $\varphi\colon X_{S}\rightarrow X$ defined as $\chi_{X\setminus\{ x \}}\mapsto x$ is a homeomorphism between $X$ and $X_{S}$.
\end{exa}
%===========================================

The following characterizes compact containment under certain conditions.

\begin{lma}\label{Hausdorff}
Let $S$ be an $\Lsc$-like \CuSgp{} with least order unit $e$. Assume that $e$ is compact and that $S$ satisfies \axiomO{5}. Then $X_S$ is normal.

Further, given $y,z\leq e$, we have $\overline{U_{y}}\subset U_{z}$ if and only if $y\ll z$.
\end{lma}
\begin{proof}
Let $x,y<e$ be two elements such that $C_{x}\cap C_{y}=C_{e}=\emptyset$. In terms of the elements in $S$, this is equivalent to $x+y\geq e\gg e$. Then, we can take $x'\ll x$ and $y'\ll y$ such that $x'+y'\gg e$.
 
 Using \axiomO{5}, there exist $c,d\leq e$ such that
 \[
  x'+c\leq e\leq x+c\quad , \quad y'+d\leq e\leq y+d .
 \]

 Consequently, we also have $x'+c+y'+d\leq e+e$ with $x'+y'\gg  e$. Since every $\Lsc$-like \CuSgp{} has weak cancellation, it follows that $c+d\leq e$.
 
 Since our order is topological, we get $c\wedge d=0$, $x\vee c= e$ and $y\vee d=e$.
 
 Using the properties listed in \autoref{prp:BasicTop}, the previous inequalities imply that $C_{x}\subset U_{c}$, $C_{y}\subset U_{d}$ and $U_{c}\cap U_{d}= U_{c\wedge d}=\emptyset$. Thus, $X_{S}$ is normal.\vspace{0.1cm}

 Now let $y,z\leq e$ and assume that $\overline{U_{y}}\subset U_{z}$, which by (iv) in  \autoref{prp:BasicTop}  happens if and only if $(y\setminus e)\vee z=e$. Further, since $y\setminus e + z\geq e$, we have $y\setminus e + z\geq e\gg e\geq y\setminus e + y$. As elements below $e$ have cancellation (see \autoref{pgr:canc}), one gets $z\gg y$.
 
 Conversely, assume that $e\geq z\gg y$. Then, by \axiomO{5} there exists an element $x$ such that $y+x\leq e$ and $e\leq z+x$.
 
 From the first inequality, it follows that $x\leq y\setminus e$, so $e\leq z+x\leq z+y\setminus e$ as required.
\end{proof}
%===========================================
In \autoref{prp:TopProp} below we study notions for \CuSgp{s} that have a topological equivalent when the semigroup is $\Lsc$-like. Recall from \autoref{pgr:AdditionalAxioms} that a \CuSgp{} is said to be \emph{countably based} if it has a countable sup-dense subset.

%===========================================
\begin{dfn}
We say that an inf-semilattice ordered \CuSgp{} $S$ is \emph{normal} if there exists an order unit $z\in S$ such that, whenever $x+ y\geq z$ for some $x,y\in S$, there exist $s,t\in S$ with
 \[
  x+ s\geq z\quad , \quad y+ t\geq z\quad , \quad s\wedge t=0.
 \]
\end{dfn}

%===========================================
\begin{prp}\label{prp:TopProp}
 Let $S$ be an $\Lsc$-like \CuSgp{} and let $X_S$ be its associated topological space. Then:
 \begin{enumerate}[(i)]
  \item $X_{S}$ is second countable if and only if $S$ is countably based.
  \item $X_{S}$ is countably compact if and only if $S$ has a compact order unit.
  \item $X_S$ is normal if and only if $S$ is normal.
  \item $X_{S}$ is a metric space whenever $S$ is countably based and normal.
  \item $X_S$ is a compact metric space whenever $S$ is  countably based, has a compact order unit and satisfies (O5).
 \end{enumerate}
\end{prp}
\begin{proof}
Let $e\in S$ be the least order unit of $S$. To show (i), assume first that $S$ is countably based with a countable basis $B$, and let $\sum '(\downarrow e)$ denote the set of finite sums of elements in $\downarrow e$. Naturally, the set $B':=B\cap \sum '(\downarrow e)$ is also a countable basis for $S$.
 
 Given an open set $U_{y}$ with $y\leq e$, write $y=\sup_{n} y_{n}$ with $y_{n}\in B'$. We have $\cup_{n} U_{y_{n}} = U_{y}$, and so $X_{S}$ is second countable.

Conversely, assume that $X_{S}$ is second countable with basis $C=\{U_{z_{n}}\}_{n}$. Note that the family $C'$ consisting of all the finite unions of sets in $C$ is also countable. 

Then, any open subset $U_{y}$ can be written as the countable union of increasing open subsets $U_{z'_{n}}$ of $C'$. We know that this is equivalent to $y=\sup_{n}z'_{n}$. This implies that the set $\sum ' \{ z'_{n}\}_{n}$ of finite sums from $\{ z'_n\}_n$ is a countable basis for $\sum ' (\downarrow e)$ and, since $\sum ' (\downarrow e)$ is dense in $S$, $\sum ' \{ z'_{n}\}_{n}$ is a countable basis for $S$.

To prove (ii), note that it is easy to check that $e\in S$ is compact if and only if $X_{S}$ is countably compact. Therefore, we are left to prove that, if there exists a compact order unit in $S$, then $e$ must also be compact.
 
 To see this, let $p$ be a compact order unit in $S$, which implies that $p\leq ne$ for some $n\in\NN$. Since we know that $p$ can be written as a finite ordered sum of elements below $e$, there exist $m\in\mathbb{N} $ and elements $q_{m}\leq \cdots\leq q_{1}\leq e$ such that $p=q_{1}+\cdots +q_{m}$.
 
 By weak cancellation applied to $q_1+\cdots +q_m\ll q_1+\cdots +q_m$, the element $q_{1}$ is  compact and satisfies $\infty q_{1}=\infty p=\infty$. Thus, $q_{1}$ is a compact order unit with $e\geq q_{1}$.  By minimality of $e$, one gets $e=q_{1}$ compact as required.
 
 Let us now show (iii). First, assume that $S$ is normal and let $z$ be the associated order unit. Let $C_x , C_y$ be closed subsets of $X_S$ with $x,y\leq e$, and recall that $C_{x}, C_{y}$ are disjoint if and only if $x\vee y=e$.
 
 Since $e$ is an order unit, we have $\infty x + \infty y= \infty \geq z$. Thus, we get $s,t$ such that $\infty x +s\geq z$, $\infty y +t\geq z$ and $s\wedge t=0$. As $z\geq e$, we know by \autoref{Sum_cup} that $\infty x\vee s\geq e$ and $\infty y\vee t\geq e$.
 
 Since $x,y\leq e$, taking the infimum with $e$ and using \autoref{lma:IntInf} we have $e=(\infty x\wedge e)\vee (s\wedge e)=x\vee (s\wedge e)$, $y\vee (t\wedge e)=e$ and $(s\wedge e)\wedge (t\wedge e)=0$. By (vi) in  \autoref{prp:BasicTop}, it follows that $C_{x}\subset U_{s\wedge e}$, $C_{y}\subset U_{t\wedge e}$ and $U_{s\wedge e}\cap U_{t\wedge e}=\emptyset$. This implies that $X_{S}$ is normal.
 
 Conversely, if $X_{S}$ is normal, it is easy to see that $S$ is normal by setting $z=e$ in the definition of normality.
 
 To prove (iv), we have that $X_{S}$ is second countable and normal by (i) and (iii), and that $C_{x}=\{x \}$ for any $x\in X_{S}$. Thus, points are closed in our topology, so $X_{S}$ is Hausdorff. We can now use Urysohn's metrization theorem to conclude that $X_{S}$ is metric (see, e.g. \cite[Theorem~34.1]{Munk2000}).
 
 For (v), note that $e\ll e$. Thus, \autoref{Hausdorff} implies that $X_S$ is normal. Following the arguments above, one gets that $X_S$ is metric. Moreover, we also know that $X_S$ is second countable and countably compact by (i) and (ii) above. Thus, $X_S$ is compact.
\end{proof}

%===========================================
We will now show that $\Lsc (X_{S},\NNbar )$ is a \CuSgp{} with the usual way-below relation for every $\Lsc$-like \CuSgp{} $S$. Note that \axiomO{1} and \axiomO{4} are always satisfied, so we are left to prove \axiomO{2} and \axiomO{3}.
%===========================================

\begin{lma} \label{wayb_comp}
Let $S$ be an $\Lsc$-like \CuSgp{} with least order unit $e$. Given $y,z\leq e$, we have $\chi_{U_{y}}\ll \chi_{U_{z}}$ in $\Lsc (X_{S},\NNbar )$ if and only if $y\ll z$ in $S$.
\end{lma}
\begin{proof}
 Assume $y\ll z$, and let $(f_{n})_{n}$ be an increasing sequence in $\Lsc (X_{S},\NNbar )$ such that $\chi_{U_{z}}\leq \sup_{n} f_{n}$. In particular, note that this holds if and only if $\chi_{U_{z}}\leq \chi_{\cup_{n} \text{supp}(f_{n})}$ or, equivalently, if 
 \[
  \bigcap_{n}(X_S\setminus\text{supp}(f_{n})) \subseteq C_{z}.
 \]
 
 Denote by $z_{n}$ the elements in $\downarrow e$ with $C_{z_{n}}=X_S\setminus\text{supp}(f_{n})$. Given that $\text{supp}(f_{n})\subseteq \text{supp}(f_{n+1})$, (i) in \autoref{prp:BasicTop} implies that $z_{n}$ is increasing.
 
 Using the proof of \autoref{topology} in the first step, we can rewrite the previous inclusion as
 \[
  C_{\sup_{n}(z_{n})} = \bigcap_{n} C_{z_n}\subseteq C_{z}.
 \]
 
Applying (i) in \autoref{prp:BasicTop} once again, one gets $z\leq \sup_{n}(z_{n})$ and, consequently, $y\leq z_{n}$ for some $n$. This implies that $U_{y}\subseteq U_{z_{n}}$ or, equivalently, that $\chi_{U_{y}}\leq \chi_{U_{z_{n}}}= \chi_{\text{supp}(f_{n})}\leq f_{n}$. This shows $\chi_{U_{y}}\ll \chi_{U_{z}}$.
\vspace{0.1cm}

Now let $y,z\leq e$ be such that $\chi_{U_{y}}\ll \chi_{U_{z}}$, and consider an increasing sequence $(h_{n})_{n}$ in $S$ such that $z\leq \sup_{n}(h_{n})$. Note that, by taking $z\wedge h_{n}$ instead of $h_{n}$, we can assume $h_{n}\leq e$ for every $n$.

Applying again (the proof of) \autoref{topology}, one gets
\[
 \bigcap_{n} C_{h_{n}}=C_{\sup_{n}(h_{n})}\subseteq C_{z},
\]
and, consequently, we have $\sup_{n}\chi_{U_{h_{n}}}\geq \chi_{U_{z}}$ since $\chi_{U_{y}}\ll \chi_{U_{z}}$, there exists some $n$ with $\chi_{U_{y}}\leq \chi_{U_{h_{n}}}$; i.e. $C_{h_n}\subseteq C_y$.

Using (i) in \autoref{prp:BasicTop} one last time, one sees that $y\leq h_{n}$ as required.
\end{proof}
%===========================================

\begin{prp}\label{LscXS_O2}
 Let $S$ be a \CuSgp{}. If $S$ is $\Lsc$-like, then $\Lsc (X_{S},\NNbar )$ satisfies \axiomO{2}.
\end{prp}
\begin{proof}
 Take $f\in \Lsc (X_{S},\NNbar )$, and let $(y_{i})_{i}$ be the sequence on $\downarrow e$ such that
 \[
  \{f\geq i\} = U_{y_{i}},
 \]
where recall that the sequence is decreasing as a consequence of (i) in \autoref{prp:BasicTop}.

Since $S$ satisfies \axiomO{2}, we have $y_{i}=\sup_{n} y_{i,n}$ with $y_{i,n}\ll y_{i,n+1}$ for every $n$.

For every fixed $k$, we have
\[
 y_{1}\geq\cdots\geq y_{k}\gg y_{k,n} \text{ for all }n.
\]

Thus, for every $i$ one can choose inductively $n_{i,k}$ with $k\geq i$ such that
\[
 y_{i,n_{i,k}}\ll y_{i,n_{i,k+1}}, \andSep y_{1,n_{1,k}}\geq\cdots\geq y_{k,n_{k,k}}.
\]

Indeed, we begin by setting $n_{1,1}=1$ (i.e. $y_{1,n_{1,1}}=y_{1,1}$). Then, assuming that we have defined $n_{i,k}$ for every $i,k\leq m-1$ (and $k\geq i$) for some fixed $m$, we set $n_{m,m}=1$, that is, $y_{m,n_{m,m}}=y_{m,1}$. We then set $n_{m-1,m}$ large enough so that $y_{m-1,n_{m-1,m}}\geq y_{m,n_{m,m}}$ and $n_{m-1,m}\geq n_{m-1,m-1}$. Similarly, we set $n_{m-2,m}\geq n_{m-2,m-1}$ such that $y_{m-2,n_{m-2,m}}\geq y_{m-1,n_{m-1,m}}$ and define $n_{i,m}$ for every $i\leq m-2$ in the same fashion.

Now consider the sums $f_{k}=\sum_{i=1}^{k} \chi_{U_{y_{i,n_{i,k}}}}$, which are ordered by construction. Thus, one has $U_{y_{i,n_{i,k}}}=\{ f_k\geq i\}$ for every $i$. Since $y_{i,n_{i,k}}\ll y_{i,n_{i,k+1}}$ for every $i$, it follows from \autoref{wayb_decomp} and \autoref{wayb_comp} that
\[
 f_{k}=\sum_{i=1}^{k} \chi_{U_{y_{i,n_{i,k}}}}\ll \sum_{i=1}^{k} \chi_{U_{y_{i,n_{i,k+1}}}} \leq f_{k+1}.
\]

It is now easy to check that $\sup_{k} f_{k}=f$.
\end{proof}
%===========================================

\begin{thm} \label{Lsc_Cu}
Let $S$ be a \CuSgp{}. 
 If $S$ is $\Lsc$-like, the monoid $\Lsc (X_{S},\NNbar )$ is a \CuSgp{}.
\end{thm}
\begin{proof}
Note that the semigroup $\Lsc (X_S ,\NNbar )$ always satisfies \axiomO{1} and \axiomO{4}. Moreover, we already know that \axiomO{2} is also satisfied by \autoref{LscXS_O2}. Thus, we are left to prove \axiomO{3}.

 Let $f\ll f'$ and $g\ll g'$. By \autoref{wayb_decomp}, this implies that
 \[
   \chi_{\{ f\geq i\}} \ll \chi_{\{ f'\geq i\}}\andSep
   \chi_{\{ g\geq i\}} \ll \chi_{\{ g'\geq i\}}
 \]
for every $i$. We also know that there exists $m<\infty$ such that $\sup (f),\sup (g)\leq m$.

Let $y_{i},y'_{i},z_{i}$ and $z'_{i}$ be elements in $\downarrow e$ such that 
\[
   U_{y_{i}}= \{ f\geq i\},\quad U_{y'_{i}}= \{ f'\geq i\},\quad 
   U_{z_{i}}= \{ g\geq i\},\andSep U_{z'_{i}}= \{ g'\geq i\}.
\]

 By \autoref{wayb_comp}, we have $y_{i}\ll y'_{i}$ and $z_{i}\ll z'_{i}$ for every $i$ and, since $S$ satisfies \axiomO{3}, one gets
 \[
  \sum_{i=1}^{m} (y_{i}+z_{i}) \ll \sum_{i=1}^{m} (y'_{i}+z'_{i}).
 \]
 
 By \autoref{std_formula}, these sums can be rewritten as
 \[
  \sum_{i=1}^{2m} \vee_{j=0}^{m}(y_{j}\wedge z_{i-j}) \ll \sum_{i=1}^{2m} \vee_{j=0}^{m}(y'_{j}\wedge z'_{i-j})
 \]
 and, since both the right and left hand side of the previous inequality are ordered, we can use \autoref{wayb_C5} to obtain
 \[
  \vee_{j=0}^{m}(y_{j}\wedge z_{i-j}) \ll \vee_{j=0}^{m}(y'_{j}\wedge z'_{i-j})
 \]
 for every $i$.
 
 Note that
 \[
  \{ f+g\geq i \} = \bigcup_{j=0}^{m}(\{ f\geq j \}\cap\{ g\geq i-j \})
 \]
so, by the equalities in the proof of \autoref{topology}, one gets 
\[
 X_S\setminus\{ f+g\geq i \} = \bigcap_{j=0}^{m} ((X_S\setminus \{ f\geq j \})\cup (X_S\setminus\{ g\geq i-j \}))
 =
 \bigcap_{j=0}^{m} (C_{y_{j}\wedge z_{i-j}})
 =
 C_{\vee_{j=0}^{m}(y_{j}\wedge z_{i-j})}
\]
 and, consequently, $\chi_{\{ f+g\geq i \}} = \chi_{U_{\vee_{j=0}^{m}(y_{j}\wedge z_{i-j})}}$.
 
The same argument also shows that $\chi_{\{ f'+g'\geq i \}}\geq \chi_{ U_{\vee_{j=0}^{m}(y'_{j}\wedge z'_{i-j})}}$.

By using \autoref{wayb_comp}, we get $\chi_{\{ f+g\geq i \}}\ll \chi_{\{ f'+g'\geq i \}}$ for every $i$. \autoref{wayb_decomp} then implies that $f+g\ll f'+g'$, as desired.
\end{proof}
%=================================================

\Msection{An abstract characterization of $\Lsc (X,\NNbar )$}{An abstract characterization of Lsc(X,N)}\label{Sect_Abst_Char}

In this section we prove that every $\Lsc$-like \CuSgp{} $S$ is $\Cu$-isomorphic to the semigroup of lower-semicontinuous functions $ \Lsc(X_{S},\NNbar )$; see \autoref{LscXS_iff_Lsc_like}. To do so, we first define a map $\varphi '\colon \Lsc(X_{S},\NNbar )_{fs} \rightarrow S$, where $\Lsc (X_{S},\NNbar )_{fs}$ is the subsemigroup of functions with finite supremum.

We then extend this map to a \CuMor{} $\varphi \colon \Lsc (X_{S},\NNbar ) \rightarrow S$. Finally, the following (probably well known) lemma will be used to complete the proof.

\begin{lma} \label{iso_red}
 Let $S,H$ be \CuSgp{s} and let $\varphi\colon S\rightarrow H$ be a \CuMor{} such that
 \begin{enumerate}[(i)]
  \item  $\varphi$ is an order embedding in a basis of $S$.
  \item  $\varphi(S)$ is a basis for $H$.
 \end{enumerate}
Then, $\varphi$ is a $\Cu$-isomorphism
\end{lma}
\begin{proof}
It is easy to see that $\varphi$ is a global order embedding.

To prove surjectivity, let $h\in H$. Since $\varphi(S)$ is a basis for $H$, we can write $h=\sup_{n} \varphi (s_{n})$ for some $s_{n}\in S$. Further, as we know that $\varphi$ is an order embedding, the sequence $(s_{n})_{n}$ is increasing in $S$, so $\sup_{n} \varphi (s_{n})=\varphi (s)$ for $s=\sup_{n} s_{n}$.
\end{proof}
%===========================================

\begin{dfn}
Let $S$ be an $\Lsc$-like \CuSgp{}. 
 Given $f\in \Lsc (X_{S},\NNbar )_{fs}$, there exists $m<\infty$ such that we can write
 \[
  f=\sum_{i=1}^{m} \chi_{\{ f\geq i \}}.
 \]
 
 We define the map $\varphi'\colon \Lsc (X_{S},\NNbar )_{fs}\rightarrow S$ as $\varphi'(f)=\sum_{i=1}^{m}z_{i}$, where $\{ f\geq i \}=U_{z_{i}}$.
\end{dfn}
%===========================================

\begin{lma}
 $\varphi'$ is a positively ordered monoid morphism  and an order embedding.
\end{lma}
\begin{proof}
 Let $f=\sum_{i=1}^{m} \chi_{U_{z_{i}}}$ as above and take $g=\sum_{j=1}^{n} \chi_{U_{y_{j}}}$. 
 We will first prove by induction on $m$ that $\varphi' (f+g)=\varphi'(f)+\varphi'(g)$.
 
 For $m=1$, $f$ is simply $\chi_{U_{z}}$ for some open subset $U_{z}$. Since $\Lsc (X_{S},\NNbar )$ is distributively lattice ordered and $0\leq 0\leq\cdots\leq 0\leq f$ is an increasing sequence, we can apply \autoref{std_formula} to get
 \[
  f + \sum_{j=1}^{n} \chi_{U_{y_{j}}} = \chi_{U_{z}\cup U_{y_{1}}} + \chi_{(U_{z}\cap U_{y_{1}})\cup U_{y_{2}} } + \cdots + \chi_{U_{z}\cap (\cap_{j} U_{y_{j}})}.
 \]
 
Applying $\varphi'$ and the equalities in the proof of \autoref{topology} at the first step, and \autoref{std_formula} at the second step, one gets
\[
 \begin{split}
 \varphi'(f+g) &= (z\vee y_{1}) + ((z\wedge y_{1})\vee y_{2}) + \cdots + (z\wedge y_{1}\wedge\cdots\wedge y_{n})\\
 &=z+(y_{1}+\cdots +y_{n}) = \varphi'(f)+ \varphi'(g).
 \end{split}
\]

Now fix any finite $m$ and assume that the result has been proven for any $k\leq m-1$. Then, using the induction hypothesis at the second step, and the case $m=1$ at the third and fourth steps, we have
\[
 \begin{split}
  \varphi '(f+g) &= \varphi' ( (f-\chi_{U_{z_{m}}}) + (g+\chi_{U_{z_{m}}}) ) = \varphi' ( f-\chi_{U_{z_{m}}}) + \varphi' (g+\chi_{U_{z_{m}}})\\
  &= \varphi' ( f-\chi_{U_{z_{m}}}) + \varphi' (g) + \varphi' (\chi_{U_{z_{m}}})= \varphi' ( f-\chi_{U_{z_{m}}}+\chi_{U_{z_{m}}}) + \varphi' (g)\\
  &= \varphi' (f) + \varphi' (g)
 \end{split}
\]
as desired.

Note that this could have also been proved using \autoref{std_formula}.

To see that $\varphi'$ is an order embedding, let $f\leq g$ in $\Lsc (X_{S},\NNbar )_{fs}$ and note that $f\leq g$ if and only if $\{ f\geq i\}\subseteq \{ g\geq i\}$ for every $i$. Let $z_i, y_i\in S$ be such that $U_{z_i}=\{ f\geq i\}$ and $U_{y_i}=\{ g\geq i\}$

By (i) in \autoref{prp:BasicTop}, $U_{z_i}\subseteq U_{y_i}$ if and only if $z_{i}\leq y_{i}$ for every $i$. Further, note that the sequences $(z_{i})_{i=1}^{m},(y_{i})_{i=1}^{m}$ are both decreasing. Since we have a topological order, $z_{i}\leq y_{i}$ for every $i$ if and only if 
\[
 \sum_{i=1}^{m} z_{i}\leq \sum_{i=1}^{m} y_{i},
\]
where note that the right and left hand side correspond to $\varphi'(f)$ and $\varphi'(g)$ respectively.
\end{proof}
%===========================================

Using $\varphi '$, one can now construct a $\Cu$-isomorphism.

\begin{thm}\label{thm:phiiso}
Let $S$ be an $\Lsc$-like \CuSgp{}. Then, the \CuMor{} $\varphi '$ extends to a $\Cu$-isomorphism.
\end{thm}
\begin{proof}
We will need the following claims.

\noindent\emph{Claim 1.} \textit{Let $(f_{n})_{n}$ and $(g_{n})_{n}$ be two $\ll$-increasing sequence with the same supremum in $\Lsc (X_S ,\NNbar )$. Then, $\sup_{n}\varphi'(f_{n})=\sup_{n}\varphi'(g_{n})$.}

Since $(f_{n})_{n}$ and $(g_{n})_{n}$ have the same supremum, we know that for every $n$ there exist $k,m$ such that $f_{n}\leq g_{m}$ and $g_{n}\leq f_{k}$.
 
 Applying $\varphi'$, we get
 \[
   \varphi'(f_{n})\leq \varphi'(g_{m})\leq \sup_{m}\varphi'(g_{m}),\andSep
   \varphi'(g_{n})\leq \varphi'(f_{k})\leq \sup_{k}\varphi'(f_{k}),
 \]
and so $\sup_{k}\varphi'(f_{k})=\sup_{m}\varphi'(g_{m})$ as desired.
\vspace{0.1cm}

By Claim 1, we can define the map $\varphi\colon \Lsc (X_{S},\NNbar )\rightarrow S$ as $\varphi(f)=\sup_{n}\varphi '(f_{n})$, where $(f_{n})_{n}$ is a $\ll$-increasing sequence with supremum $f$. We will see that $\varphi$ is a \CuMor{} that extends $\varphi '$ and that it satisfies the conditions in \autoref{iso_red} (i.e. $\varphi$ is a $\Cu$-isomorphism).\vspace{0.1cm}

\noindent\emph{Claim 2.} \textit{Let $(f_{n})_{n}$ be an increasing sequence in $\Lsc (X_{S},\NNbar )_{fs}$ with supremum $f=\sup f_n\in \Lsc (X_{S},\NNbar )_{fs}$. Then, we have $\varphi'(f)=\sup_{n}\varphi'(f_{n})$.}

To prove the claim, let $f_{n}=\chi_{U_{z_{n}}}$ for every $n$, and recall that $\sup_{n}\chi_{U_{z_{n}}}=\chi_{U_{\sup_{n}(z_{n})}}$. This is equivalent to 
 \[
  \varphi'(\sup_{n}f_{n})=\sup_{n}(z_{n})=\sup_{n}\varphi'(f_{n})
 \]
Now, given any increasing sequence as in the statement of the lemma with supremum $f$, we know that $\sup (f)<\infty$, say $\sup (f)=m\in\NN$.

Thus, given $U_{i,n}=\{f_{n}\geq i\}$ for $1\leq i\leq m$, we can write
\[
 f_{n}=\sum_{i=1}^{m} \chi_{U_{i,n}}
\]
with some possibly empty $U_{i,n}$'s.

We have
\[
 f=\sup_{n} (f_{n})=\sum_{i=1}^{m} \chi_{\cup_{n}U_{i,n}},
\]
where $\cup_{n}U_{i,n}=\{ f\geq i\}$.

Using that $\varphi'$ preserves suprema of indicator functions, we have
\[
 \varphi'(f)= \sum_{i=1}^{m} \varphi'(\chi_{\cup_{n}U_{i,n}})=
 \sup_{n}\varphi'(\sum_{i=1}^{m}\chi_{U_{i,n}})=\sup_{n}\varphi'(f_{n}),
\]
as required.\vspace{0.1cm}

 Since $\varphi'$ preserves addition and $S$ is a $\Cu$-semigroup, it is clear that $\varphi$ also preserves addition. Note that the proof of the Claim 2 above also shows that $\varphi$ extends $\varphi '$ and that $\varphi$ is order preserving.\vspace{0.1cm}

To see that $\varphi$ preserves suprema, let $\varphi (f)=\sup_{n}\varphi'(f_{n})$ with $(f_{n})_{n}$ $\ll$-increasing with supremum $f$ and consider an increasing sequence $(g_{n})_{n}$ whose supremum is also $f$.

Then, for every $n$ there exists an $m$ with $f_{n}\leq g_{m}$ and, consequently, $\varphi'(f_{n})=\varphi(f_{n})\leq \varphi(g_{m})$. It follows that $\varphi (f)=\sup_{n}\varphi'(f_{n})\leq \sup_{m}\varphi(g_{m})$.

On the other hand, $f\geq g_{m}$ for every $m$, so $\varphi(f)=\sup_{m}\varphi(g_{m})$.\vspace{0.1cm}

Now let $f,g\in \Lsc (X_{S},\NNbar )$ be such that $f \ll g$. Then, we know by \autoref{wayb_decomp} that this happens if and only if
\[
  \chi_{\{ f\geq i\}}\ll \chi_{\{ g\geq i\}} \text{ for every $i$} \andSep \sup (f)=m<\infty.
 \]
 
Letting $y_i, z_i\in S$ such that $U_{y_{i}}= \{ f\geq i\}$ and $U_{z_{i}}= \{ g\geq i\}$, we know that
\[
 \varphi(f)=\varphi'(f)= y_{1}+\cdots +y_{m}\ll z_{1}+\cdots +z_{m}\leq \sup_{n}\sum_{i=1}^{n} z_{i}=\varphi(g).
\]

Finally, note that the image of $\varphi$ is clearly dense in $S$, since $\varphi'$ is surjective on $\downarrow e$. Further, $\varphi$ is an order embedding in $\Lsc (X_{S},\NNbar )_{fs}$, since $\varphi$ coincides with $\varphi'$ in this basis of $\Lsc (X_{S},\NNbar )$.

Thus, since the conditions in \autoref{iso_red} are satisfied, it follows that $\varphi$ is a $\Cu$-isomorphism.
\end{proof}
%===========================================

\begin{thm}\label{LscXS_iff_Lsc_like}
 Let $S$ be a $\Cu$-semigroup. Then, $S$ is $\Lsc$-like if and only if $S$ is $\Cu$-isomorphic to $\Lsc (X,\NNbar )$ for a $T_{1}$ topological space $X$.
\end{thm}
\begin{proof}
We already know that $\Lsc (X,\NNbar )$ is $\Lsc $-like  whenever it is a \CuSgp{} (see \autoref{exa_Lsc_is_Lsc}), and the converse follows from \autoref{thm:phiiso}.
\end{proof}

%===========================================

In \cite{ThiVil21arX:DimCu} a notion of covering dimension for \CuSgp{s} is introduced. This dimension satisfies many of the expected permanence properties (\cite[Proposition~3.10]{ThiVil21arX:DimCu}), and is related to other dimensions, such as the nuclear dimension of \ca{s} (\cite[Theorem~4.1]{ThiVil21arX:DimCu}) and the  Lebesgue covering dimension (see \cite[Proposition~4.3, Corollary~4.4]{ThiVil21arX:DimCu}).

Using such a notion, one can prove the following.

\begin{thm}\label{thm:dimChar}
 Let $S$ be a \CuSgp{} satisfying \axiomO{5} and let $n\in\NN\cup\{ \infty\}$. Then, $S$ is $\Cu$-isomorphic to $\Lsc (X,\NNbar )$ with $X$ a compact metric space such that $\dim (X)= n$ if and only if $S$ is $\Lsc$-like, countably based,  has a compact order unit, and $\dim (S)=n$.
 
 In particular, a \CuSgp{} $S$ is $\Cu$-isomorphic to the Cuntz semigroup of $C(X)$ with $X$ compact metric and $\dim (X)\leq 1$ if and only if $S$ is $\Lsc$-like, countably based, satisfies \axiomO{5}, has a compact order unit, and $\dim (S)\leq 1$.
\end{thm}
\begin{proof}
The forward implication follows from Examples  \ref{exa_Lsc_is_Lsc} and \ref{exa:LscSpace}, \autoref{prp:TopProp} and \cite[Corollary~4.4]{ThiVil21arX:DimCu}.

To prove the converse, use \autoref{LscXS_iff_Lsc_like} and (v) in \autoref{prp:TopProp} to deduce that $S\cong\Lsc (X_S ,\NNbar )$ with $X_S$ compact metric. Then, it follows from \cite[Corollary~4.4]{ThiVil21arX:DimCu} that $\dim (X_S)=\dim (S)=n$, as required.

Now assume that $S$ is $\Cu$-isomorphic to the Cuntz semigroup of $C(X)$ with $X$ compact metric and $\dim (X)\leq 1$. By \cite[Theorem~1.1]{Rob13CuSpDim2} we know that $\Cu (C(X))\cong \Lsc (X,\NNbar )$. In particular, $S$ satisfies \axiomO{5} (e.g. \cite{RorWin10ZRevisited}). 

Thus, it now follows from our previous argument that $S$ is $\Lsc$-like, countably based, satisfies \axiomO{5}, has a compact order unit and $\dim (S)\leq 1$.

Conversely, if $S$ satisfies the list of properties in the second part of the statement, note that $S\cong \Lsc (X_S ,\NNbar )$ with $\dim (X_S)\leq 1$ again by \autoref{LscXS_iff_Lsc_like}. Using  \cite[Theorem~1.1]{Rob13CuSpDim2} a second time, it follows that $\Lsc (X_S ,\NNbar )\cong \Cu (C(X_S))$, as desired.
\end{proof}
%===========================================

\section{Chain conditions and the Cuntz semigroup of commutative AI-algebras}\label{Sct_proper}

In this section, we introduce the notions of piecewise chainable and weakly chainable \CuSgp{s} and prove that, together with some additional properties, these notions give a characterization of when $S$ is $\Cu$-isomorphic to the Cuntz semigroup of a unital block stable AI-algebra and a unital AI-algebra respectively; see \autoref{char_piecewise} and \autoref{almost_chainable_Cu}.

We also show that the Cuntz semigroup of any AI-algebra is weakly chainable, thus uncovering a new propery that the Cuntz semigroup of any AI-algebra satisfies; see \autoref{cor:WeakCha}.

%===========================================

We first prove the following categorical proposition, which summarizes the results of the above sections. We denote by $\CatTop$ the category of topological spaces.

\begin{prp}
Let $\mathcal{T}_{1}$ denote the subcategory of $\CatTop$ whose objects are the $T_{1}$ topological spaces, and let $\Lsc$ be the subcategory of $\Cu$ consisting of $\Lsc$-like \CuSgp{s}.

 There exists a faithful and essentially surjective  contravariant functor $T\colon \mathcal{T}_{1}\to \Lsc$ that is full on isomorphisms.
\end{prp}
\begin{proof}
 For every topological space $X\in\mathcal{T}_{1}$, define $T(X)=\Lsc (X,\NNbar )$. 
 
 Also, given any continuous map $f\colon X\to Y$, set $T(f)\colon \Lsc (Y,\NNbar )\to \Lsc (X,\NNbar )$ as the unique $\Cu$-morphism such that $T(f)(\chi_{U})=\chi_{f^{-1}(U)}$ for every open subset $U$ of $Y$.
 
 Note that, given $f\colon X\to Y$ and $g\colon Y\to Z$ in $\mathcal{T}_{1}$, we have 
 \[
 T(g\circ f)(\chi_{U})=\chi_{(g\circ f)^{-1}(U)}=\chi_{f^{-1}g^{-1}(U)}=(T(f)\circ T(g))(\chi_{U}).
 \]

 Thus, $T$ is a contravariant functor, which is clearly faithful by construction.
 
 Moreover, we know by \autoref{LscXS_iff_Lsc_like} that for every $\Lsc$-like \CuSgp{} $S$ there exists a $T_{1}$ space $X_{S}$ with $S\cong \Lsc (X_{S},\NNbar )$. Therefore, $T$ is essentially surjective.
 \vspace{0.1cm}
 
 Now let $\varphi\colon S\to T$ be a $\Cu$-isomorphism of $\Lsc$-like \CuSgp{s}. Using \autoref{LscXS_iff_Lsc_like}, we get a $\Cu$-isomorphism of the form $\phi\colon\Lsc (X_{S},\NNbar )\to \Lsc (X_{T},\NNbar )$.
 
 Since $\phi (1)= 1$, indicator functions must map to indicator functions. Since $\phi $ is a $\Cu$-isomorphism, maximal elements below $1$ must map to maximal elements below $1$. More explicitly, for every $x\in X_{S}$, there exists $y\in X_{T}$ such that $\phi (\chi_{X_{S}\setminus\{ x\}})=\chi_{X_{T}\setminus\{ y\}}$.
 
 We define the map $f\colon X_{T}\to X_{S}$ as $y\mapsto x$, which is bijective because $\phi$ is a $\Cu$-isomorphism.
 
 To see that it is continuous, let $U$ be an open subset of $X_{S}$ and let $V\subset X_{T}$ be such that $\phi (\chi_{U})=\chi_{V}$. Then, given $y\in X_{T}$, we have that $y\in V$ if and only if
 \[
  1\leq \chi_{X_{T}\setminus\{ y\}}+\chi_{V}=
  \phi (\chi_{X_{S}\setminus\{ f(y)\}}+\chi_{U}).
 \]

 Since $\phi$ is a $\Cu$-isomorphism, this in turn holds if and only if $(X_{S}\setminus\{ f(y)\})\cup U=X_{S}$ or, equivalently, if $f(y)\in U$.
 
 This shows that $f^{-1}(U)=V$ and, consequently, that $f$ is continuous.
 
 Finally, let $V\subset X_{T}$ be open. Since $\phi$ is an isomorphism, there exists some open subset $U\subset X_{S}$ such that $\chi_{V}=\phi (\chi_{U})$.
 
 By the argument above, one has $V=f^{-1}(U)$ and, since $f$ is bijective, it follows that $f(V)=f(f^{-1}(U))=U$. This shows that $f$ is open.
 
 Thus, $f$ is a homeomorphism between $X_{S}$ and $X_{T}$, as required.
\end{proof}

%===========================================

We now introduce chainable and piecewise chainable inf-semilattice ordered \CuSgp{s}.

\begin{dfn}\label{dfn:ChaCuSgp}
Let $S$ be an inf-semilattice ordered \CuSgp{}. An element $x\in S$ is said to be \emph{chainable} if for every sum $y_{1}+\cdots +y_{n}\geq x$, there exist elements $z_{1},\cdots ,z_{m}$ such that
 \begin{enumerate}[(i)]
  \item For every $i$ there exists some $k$ with $z_{i}\leq y_{k}$
  \item $z_{i}\wedge z_{j}\neq 0$ if and only if $\vert i-j \vert \leq 1$
  \item $z_{1}+\cdots +z_{m}\geq x$
 \end{enumerate}
 
 $S$ will be called \emph{chainable} if it has a chainable order unit.
 
 Moreover, we will say that $S$ is \emph{piecewise chainable} if there exist chainable elements $s_{1},\cdots ,s_{n}$ such that $s_{1}+\cdots +s_{n}$ is an order unit and $s_{i}\wedge s_{j}=0$ whenever $i\neq j$
\end{dfn}
%===========================================

\begin{lma}\label{Chainability} 
 Given an $\Lsc$-like \CuSgp{} $S$ with least order unit $e$ and an element $y\leq e$, $U_{y}$ is topologically chainable if and only if $y$ is chainable.
 
 In particular, $S$ is chainable if and only if $X_{S}$ is topologically chainable.
\end{lma}
\begin{proof}
If $y$ is chainable, take a finite cover $U_{y_{1}}\cup\cdots\cup U_{y_{n}}=U_{y}$. We have that $y=y_{1}\vee\cdots\vee y_{n}\leq y_{1}+\cdots +y_{n}$. Thus, applying the chainability of $y$, one gets elements $z_{1}, \cdots , z_{m}$ such that for every $i$ there exists $k$ with $z_{i}\leq y_{k}\leq y$. This shows that $z_1\vee \cdots \vee z_m\leq y$.

By \autoref{Sum_cup} and (iii) in \autoref{dfn:ChaCuSgp}, we have $z_{1}\vee\cdots\vee z_{m}\geq y$ and, consequently, $z_1\vee \cdots \vee z_m= y$. This shows that $U_{z_{1}},\cdots, U_{z_{m}}$ is a cover for $U_y$.

Using the equalities in the proof of \autoref{topology} and conditions (i)-(iii) in \autoref{dfn:ChaCuSgp}, one sees that $U_{z_{1}},\cdots, U_{z_{m}}$ is a chain that refines our original cover in the sense of \autoref{dfn_chainable}.

Conversely, if $U_{y}$ is topologically chainable and we have a sum $y_{1}+\cdots +y_{n}\geq y$, we can apply \autoref{Sum_cup} once again to obtain
\[
 (y_{1}\wedge y)\vee\cdots \vee (y_{n}\wedge y)=y.
\]

This shows that $U_{y_{1}\wedge y}\cup\cdots\cup U_{y_{n}\wedge y}=U_y$, and we can use the chainability of $U_{y}$ to obtain a chain refining this cover. Using \autoref{topology}, it is easy to check that the elements below $e$ corresponding to the open subsets of the chain satisfy conditions (i)-(iii) in \autoref{dfn:ChaCuSgp}.

 In particular, the previous argument shows that $e$ is chainable whenever $X_{S}$ is topologically chainable. By definition, this implies that $S$ is chainable.
 
 Conversely, if $S$ is chainable, we have a chainable order unit $s$. Let us now show that $e$ is also chainable, which by the above arguments will imply that  $U_e=X_{S}$ is topologically chainable.
 
 Thus, let $y_1+\cdots +y_n\geq e$, which by \autoref{Sum_cup} implies that $y_1\vee\cdots\vee y_n\geq e$. Since $s$ is an order unit, one has
 \[
  \infty y_1\vee\cdots\vee \infty y_n\geq \infty e = \infty = \infty s\geq s.
 \]

 Using that $s$ is chainable, we obtain elements $z_1,\cdots ,z_m$ satisfying (i)-(iii) in \autoref{dfn:ChaCuSgp}. In particular, since for every $i$ there exists $k$ with $z_i\leq \infty y_k$, one can use \autoref{lma:IntInf} in the second step to get
 \[
  z_i\wedge e\leq (\infty y_k)\wedge e= y_k\wedge e\leq y_k
 \]

 Further, since $e$ is the least order unit in $S$ and $z_1+\cdots +z_m\geq s\geq e$, it follows from \autoref{Sum_cup} that $z_1\vee \cdots \vee z_m\geq e$. Taking the infimum by $e$ and using \autoref{Sum_cup} once again, we get $z_1\wedge e+\cdots +z_m\wedge e\geq e$.
 
 This shows that the elements $z_i\wedge e$ satisfy conditions (i)-(iii) in \autoref{dfn:ChaCuSgp} for $y_1\vee\cdots\vee y_n\geq e$, as desired.
\end{proof}
%===========================================

\begin{lma}\label{P_W_Chainable}
 A countably based $\Lsc$-like \CuSgp{} $S$ with a compact order unit is piecewise chainable if and only if $X_{S}$ is.
\end{lma}
\begin{proof}
 If $X_{S}$ is piecewise chainable, there exist chainable components $Y_{1},\cdots , Y_{n}$ such that $X_{S}=Y_{1}\sqcup\cdots\sqcup Y_{n}$. Since chainability implies connectedness (whenever the space is compact), there is a finite number of connected components, and so these are clopen.
 
 By \autoref{Chainability}, the disjoint chainable components correspond to disjoint chainable elements, so $S$ is piecewise chainable by definition.
 \vspace{0.1cm}
 
 Conversely, if $S$ is piecewise chainable, each element $s_{i}$ in the definition of chainable corresponds to a chainable open subset of $X_{S}$, which is disjoint from the other chainable open subsets by construction.
\end{proof}
%===========================================

\begin{thm}\label{char_piecewise}
Let $S$ be a \CuSgp{}. Then, $S$ is $\Cu$-isomorphic to the Cuntz semigroup of a unital block-stable commutative AI-algebra if and only if $S$ is countably based, $\Lsc$-like, piecewise chainable, has a compact order unit, and satisfies \axiomO{5}.
\end{thm}
\begin{proof}
 Let $S$ be $\Cu$-isomorphic to the  Cuntz semigroup of a unital commutative block stable AI-algebra. Then, we know from \cite[Theorem~1.1]{Rob13CuSpDim2} and \autoref{dfn:PieceWiseChain} that $S\cong \Lsc (X,\overline{\NN})$ with $X$ a compact, metric, piecewise chainable space. In particular, $S$ satisfies \axiomO{5}, has a compact order unit, is countably based and $\Lsc$-like. Using \autoref{P_W_Chainable}, it also follows that $S$ is piecewise chainable.
 
 Conversely, assume that $S$ satisfies all the conditions in the list. By \autoref{LscXS_iff_Lsc_like} and (v) in  \autoref{prp:TopProp}, we have $S\cong\Lsc (X,\NNbar )$ with $X$ a compact metric space.
 
 Then, it follows from \autoref{exa:LscSpace} and \autoref{P_W_Chainable} that $X$ is piecewise chainable. In particular, it has dimension less than or equal to one by \autoref{rmk:ChainImplDim1}. Thus, $\Cu (C(X))\cong \Lsc (X,\NNbar )$ by \cite[Theorem~1.1]{Rob13CuSpDim2}, so $S$ is isomorphic to the Cuntz semigroup of a unital commutative block-stable AI-algebra.
\end{proof}
%===========================================

We now define weak chainability for any \CuSgp{} and prove that every \CuSgp{} of an AI-algebra satisfies such a condition. Moreover, we also show that an $\Lsc$-like \CuSgp{} is weakly chainable if and only if its associated space is almost chainable.

Given two elements $x,y$ in a \CuSgp{}, we write $x\propto y$ if there exists $n\in\NN$ with $x\leq ny$.

\begin{dfn}\label{wch}
  We will say that a \CuSgp{} $S$ is \emph{weakly chainable}, or that it satisfies the \emph{weak chainability condition} if, for any $x,y,y_{1},\cdots ,y_{n}$ such that
  \[
  x\ll y\ll y_1+\cdots +y_n,
 \]
there exist $x', z_{1},\cdots ,z_{m}\in S$ such that $x'\leq y$, $x\propto x'$ and
  \begin{enumerate}[(i)]
   \item For any $i$ there exists $j$ such that $z_{i}\leq y_{j}$.
   \item $z_{i}+z_{j}\leq x'$ whenever $\vert i-j\vert \geq 2$.
   \item $z_{1}+\cdots +z_{m}\geq x'$.
  \end{enumerate}
\end{dfn}
%===========================================

\begin{lma}
  Let $X$ be a compact metric space. Then, $\Lsc (X,\NNbar )$ is weakly chainable if and only if $X$ is almost chainable.
 \end{lma}
 \begin{proof}
  First, recall that $S=\Lsc (X,\NNbar )$ is a \CuSgp{} whenever $X$ is compact and metric by \autoref{cor:LscXCu}. Further, also recall from \autoref{exa_Lsc_is_Lsc} that $S$ is an $\Lsc$-like \CuSgp{} with least order unit $1$, and that $X\cong X_S$ by \autoref{exa:LscSpace}.
 
  Now assume that $S$ is weakly chainable, and let $U_{y_1},\cdots , U_{y_n}$ be a cover of $X_S$. Then, the elements $y_1,\cdots ,y_n\leq 1$ satisfy $y_1\vee \cdots \vee y_n\geq 1$ and, consequently, $y_1+\cdots +y_n\geq 1$.
  
  Set $x=y=1$, and apply \autoref{wch} to obtain elements $x',z_1,\cdots ,z_m$ satisfying the conditions in the definition. Note that, since $x'$ satisfies $x'\leq 1\leq kx'$ for some $k\in\mathbb{N}$, it follows from the second inequality that $x'\geq 1$ and, therefore, $x'=1$. Let $U_{z_1},\cdots ,U_{z_m}$ be the open subsets of $X_S$ corresponding to $z_1,\cdots ,z_m$ respectively. Using (i)-(iii) in \autoref{wch}, it is easy to see that such sets form an almost chain refining the original cover. This implies that $X_S$ is almost chainable and, since $X\cong X_S$, so is $X$.
  
  Conversely, assume that $X$ is almost chainable and let $x,y,y_{1},\cdots ,y_{n}\in S$ be as in \autoref{wch}. Set $x'=x\wedge 1$. Then, we know by \autoref{Sum_cup} that $y_{1}\vee\cdots\vee y_{n}\geq y\gg x'$. Taking the infimum with $1$, one has $(y_{1}\wedge 1)\vee\cdots\vee (y_{n}\wedge 1)\gg x'$.
  
  Moreover, note that $x'$ satisfies $x'\leq y$ and $x\propto x'$. Since $X$ is compact, metric and almost chainable, we know by \autoref{rmk:ChainImplDim1} that its dimension is less than $2$. This implies by \cite[Theorem~1.1]{Rob13CuSpDim2} that $\Cu (C(X))\cong\Lsc (X,\overline{\mathbb{N}})$ and, in particular, that $S$ satisfies \axiomO{5}.
  
  Thus, by the proof of \autoref{topology} and  \autoref{Hausdorff}, $(y_{1}\wedge 1)\vee\cdots\vee (y_{n}\wedge 1)\gg x'$ corresponds to a cover $\overline{U_{x'}}\subset U_{y_{1}}\cup\cdots\cup U_{y_{n}}$. In particular, the open sets $X\setminus\overline{U_{x'}},U_{y_{1}},\cdots , U_{y_{n}}$ form a cover of $X$ and, since $X$ is almost chainable, there exists an almost chain $C_{1},\cdots , C_{m}$ covering $X$ and refining $X\setminus\overline{U_{x'}},U_{y_{1}},\cdots , U_{y_{n}}$.
  
  Now take the almost chain $C_{1}\cap U_{x'},\cdots , C_{m}\cap U_{x'}$, which clearly covers $U_{x'}$. For each $i$, let $z_{i}\in S$ be the associated element to $C_{i}\cap U_{x'}$. These elements satisfy the desired conditions in \autoref{wch}.
  
  Indeed, to see condition (ii) note that $z_i\leq x'$ for every $i$, so it follows that $z_i+z_j\leq x'$ if and only if $z_i\wedge z_j=0$. By the proof of \autoref{topology}, this is equivalent to $(C_{i}\cap U_{x'})\cap (C_{i}\cap U_{x'})=\emptyset$. Since $\{C_{i}\cap U_{x'}\}_{i}$ is an almost chain, this condition is satisfied.
  
  Conditions (i) and (iii) follow similarly using that $\{C_{i}\cap U_{x'}\}_{i}$ is a cover of $U_{x'}$ refining $\{U_{y_{j}}\}_{j}$.
\end{proof}
%=======================================================

Using \autoref{rmk:ChainImplDim1}, one gets the following result. 

\begin{cor}\label{connected_wch}
  Given $X$ compact, metric and connected, the \CuSgp{} $\Lsc (X,\NNbar )$ is weakly chainable if and only if $X$ is chainable.
\end{cor} 
%===========================================

It would be interesting to know whether a general $\Lsc$-like \CuSgp{} $S$ is weakly chainable if $X_S$ is almost chainable.
%===========================================

\begin{lma}\label{lma:WeakChSums}
 Given two weakly chainable \CuSgp{s}  $S$ and $T$, their direct sum $S\oplus T$ is also weakly chainable.
\end{lma}
\begin{proof}
 Take $x,y,y_{1},\cdots, y_{n}\in S\oplus T$ as in \autoref{wch}. Write $x=(x_{1},x_{2})$, $y=(y_{1},y_{2})$ and $y_{j}= (y_{j,1},y_{j,2})$ with $x_{1},y_{1},y_{j,1}\in S$ and $x_{2},y_{2},y_{j,2}\in T$.

Since $S$ and $T$ are weakly chainable, one gets elements $x'_{1}$, $z_{1,1}$, $\cdots$ , $z_{m,1}\in S$ and $x'_{2}$, $z_{1,2}$, $\cdots$, $z_{m',2}\in S$ satisfying the conditions in  \autoref{wch}. Define $x'=(x'_{1},x'_{2})$ and note that $x'\leq y$ and that there exists some $k\in\mathbb{N}$ with  $x\leq kx'$.

Now set $z_{i}=(z_{i,1},0)$ for $i\leq m$ and $z_{i}=(0,z_{i-m+1,2})$ for $i> m$. We have that:
\[
\begin{split}
 z_{1}+\cdots +z_{m+m'-1}&= ((z_{1,1},0)+\cdots +(z_{m,1},0)) + ((0,z_{1,2})+\cdots +(0,z_{m',2}))\\
 &\geq (x'_{1},0) + (0,x'_{2})=x'.
 \end{split}
\]

As expected, we also get that for every $i$ there exists a $j$ such that $z_{i}\leq (y_{j,1},0)\leq y_{j}$ or $z_{i}\leq (0,y_{j,2})\leq y_{j}$.

Now take $z_{i},z_{j}$ with $\vert i-j\vert \geq 2$. If $i,j\leq m$, we have $z_{i}+z_{j}\leq (x'_{1},0)\leq x'$. Similarly, $z_{i}+z_{j}\leq (0,x'_{2})\leq x'$ whenever $i,j> m$.

Moreover, if $i\leq m$ and $j>m$, we know that $z_{i}=(z_{i,1},0)\leq (x'_{1},0)$ and $z_{j}=(0,z_{j-m+1,2})\leq (0,x'_{2})$. This implies $z_{i}+z_{j}\leq (x'_{1},0) + (0,x'_{2})=x'$.

Since $z_{1},\cdots , z_{m+m'-1}$ satisfy all the required properties, $S$ satisfies the weak chainability condition.
\end{proof}
%=======================================================

\begin{prp}\label{prp:Wchindlim}
 The weak chainability condition goes through inductive limits.
\end{prp}
\begin{proof}
 Let $S=\lim S_{n}$ with $S_{n}$ weakly chainable for every $n$. Given an element $x\in S_n$, let us denote its image through the canonical map $S_n\to S$ by $[x]$.
 
 Let $x,y,y_{1},\cdots ,y_{n}\in S$ be as in \autoref{wch}. Then, let $m\in\NN$ be such that there exist elements $u, v$ and $v_j$ in $S_m$ with $[v_j]\ll y_j$,
 \[
 x\ll [u]\ll [v]\ll y\ll [v_1]+\cdots +[v_n]\ll y_1+\cdots +y_n,
 \]
and $u\ll v\ll v_1+\cdots +v_n$.

Since $S_{m}$ is weakly chainable, we obtain elements $u', z_{1},\cdots ,z_{m}\in S_{m}$ satisfying the conditions in \autoref{wch}.  We have:
\begin{enumerate}[(i)]
  \item $u'\leq v , u\leq ku'$ for some $k\in\mathbb{N}$. This implies $ [ u' ]\leq [ v ]\leq y$ and $x\leq [ u ]\leq k [ u' ]$.
   \item For any $i$ there exists $j$ such that $z_{i}\leq v_{j}$, which shows that $[ z_{i} ]\leq [ v_{j} ] \leq y_j$.
   \item $z_{i}+z_{j}\leq u'$ whenever $\vert i-j\vert \geq 2$. Consequently, $[ z_{i} ] + [ z_{j} ]\leq [ u' ]$ whenever $\vert i-j\vert \geq 2$.
\end{enumerate}
  Since $z_{1}+\cdots +z_{m}\geq u'$, one also gets $ [u']\leq [z_{1}]+\cdots + [z_{m}]$. Thus, $S$ is weakly chainable, as desired.
\end{proof}

\begin{cor}\label{cor:WeakCha}
 The Cuntz semigroup of any AI-algebra is weakly chainable.
\end{cor}
%=======================================================

\begin{exa}\label{exa:LscTnotwch}
 The \CuSgp{s} $\Lsc (\mathbb{T},\NNbar )$ and $\Lsc ([0,1]^{2},\NNbar )$ do not satisfy the weak chainability condition.
 
 Indeed, this follows clearly from \autoref{connected_wch}, as $\mathbb{T}$ and $[0,1]^{2}$ are not chainable continua.
\end{exa}
%=======================================================

Using the results developed thus far, one can now use an analogous proof to that of \autoref{char_piecewise} to prove the following theorem.

\begin{thm}\label{almost_chainable_Cu}
 Let $S$ be a \CuSgp{}. Then, $S$ is $\Cu$-isomorphic to the Cuntz semigroup of a unital commutative AI-algebra if and only if $S$ is countably based, $\Lsc$-like,  weakly chainable, has a compact order unit, and satisfies \axiomO{5}.
\end{thm}
%===========================================

\section{New properties of the Cuntz semigroup of an AI-algebra}\label{Sect_Proper}

Inspired by the abstract characterization obtained above, in this section we introduce properties that are satisfied by the Cuntz semigroups of all AI-algebras and that are not satisfied by other well known $\Cu$-semigroups. In \autoref{wch} we have already introduced one such property, which is not satisfied by $\Lsc (\mathbb{T},\NNbar )$; see \autoref{exa:LscTnotwch}. We now introduce the conditions of \CuSgp{s} with refinable sums and almost ordered sums, and show that the Cuntz semigroup of all AI-algebras satisfy these properties. We also prove that $Z$, the Cuntz semigroup of the Jiang-Su algebra $\mathcal{Z}$, does not have refinable sums; see \autoref{exa:ZnotRefSum}.
%===========================================

\begin{dfn}\label{dfn_Rs}
 We say that a \CuSgp{} $S$ has \emph{refinable sums} if, given a finite $\ll$-increasing sequence
 \[
  x_{1}\ll \cdots\ll x_{n}
 \]
 and elements $x'_{1},\cdots ,x'_{n}$ such that $x_{i}\propto x'_{i}$ for every $i$, there exist finite decreasing sequences $(y_{j}^{i})_{j=1}^{l}$ such that:
\begin{enumerate}[(i)]
 \item $x'_{i+1}\geq y_{1}^{i}$ for every $i$.
 \item $y_{j}^{i}\ll y_{j}^{i+1}$ for every $i$ and $j$.
 \item $x_{i}\ll y_{1}^{i}+\cdots +y_{l}^{i} \ll x_{i+1}$.
\end{enumerate}
\end{dfn}
%===========================================

\begin{exa}\label{exa:LscRefSum}
 Any \CuSgp{} $S$ of the form $\Lsc (X,\NNbar )$ has refinable sums.
 
 To see this, let $x_{i},x'_{i}$ as in \autoref{dfn_Rs}, and let $\tilde{x_{i}}$ be such that
 \[
  x_{1}\ll \tilde{x}_{1} \ll x_{2}\ll \tilde{x}_{2} \ll x_{3}\ll \cdots\ll x_{n}.
 \]

 Since $\tilde{x}_{i}\in S_{\ll}$ for each $i$, they can all be written as ordered finite sum of elements below one. Further, by possibly adding some zeros, we may assume that all $\tilde{x}_{i}$'s have the same amount of summands. That is to say, we have
 \[
  x_{1}\ll \tilde{x}_{1}=y_{1}^{1}+\cdots +y_{l}^{1} \ll x_{2}\ll \tilde{x}_{2}=y_{1}^{2}+\cdots +y_{l}^{2} \ll x_{3}\ll\cdots\ll x_{n}.
 \]
 
Thus, we know by \autoref{wayb_C5} that $y_{j}^{i}\ll y_{j}^{i+1}$ for every $i,j$. Moreover, since we have $\tilde{x}_{i}\ll x_{i+1}\propto x_{i+1}'$, one can find $x''_{i+1}\ll x'_{i+1}$ such that $\tilde{x}_{i}\propto x''_{i+1}$.

Since $x''_{i+1}\in S_{\ll}$, we have that $\tilde{x}_{i}\propto x''_{i+1}\propto x''_{i+1}\wedge 1$. Applying the topological order in $S$, we obtain $y_{1}^{i}\leq x''_{i+1}\wedge 1\leq x'_{i+1}$, as required.
\end{exa}
%===========================================

\begin{exa}\label{exa:ZnotRefSum}
 Let $Z=(0,\infty ]\sqcup \NN$, and denote by $n'\in (0,\infty ]$ the associated element to $n\in\NN$. Order and addition in $Z$ are defined normally in each component, and given $x\in (0,\infty ]$ and $n\in\NN$ we set $x\leq n$ if and only if $x\leq n'$; $n\leq x$ if and only if $n'<x$; and $x+n=x+n'$. It was proved in \cite[Theorem~3.1]{PerTom07Recasting} that $Z$ is $\Cu$-isomorphic to the Cuntz semigroup of the Jiang-Su algebra $\mathcal{Z}$, as defined in \cite{JiaSu99}.
 
 We claim that $Z$ does not have refinable sums. Indeed, set $x_{1}=x'_{1}=x_{2}=x'_{2}=1$, $x_{3}=1.1$ and $x'_{3}=0.5$. Then, if $Z$ were to have refinable sums, we would get
 \[
  1\ll y_{1}^{1}+\cdots +y_{l}^{1}\ll 1\ll y_{1}^{2}+\cdots +y_{l}^{2}\ll 1.1 .
 \]

 It follows, in particular, that $y_{1}^{1}=1$ and $y_{i}^{1}=0$ whenever $i\geq 2$. 
 
 This, together with the properties of our sums,  implies that
 \[
  1=y_{1}^{1}\ll y_{1}^{2}\leq 0.5,
 \]
which is a clear contradiction.
\end{exa}
%===========================================

\begin{prp}\label{prp:RefSum}
 Let $S$ be a \CuSgp{} that can be written as an inductive limit $\lim S_{k}$ of \CuSgp{s} $S_{k}$ that have refinable sums. Then, $S$ also has refinable sums.
\end{prp}
\begin{proof}
Let $S=\lim S_{k}$ where each $S_{k}$ has refinable sums. As in the proof of \autoref{prp:Wchindlim}, let us denote the image through the canonical map $S_k\to S$ of an element $x\in S_k$ by $[x]$.

 Let $x_{1},\cdots , x_{n}$ and $x'_{1},\cdots , x'_{n}$ be elements in $S$ as in \autoref{dfn_Rs}. Let $k\in\NN$ such that, for every $i\leq n-1$, there exist elements $u_{2i-1},u_{2i},v_{2i-1},v_{2i}\in S_k$ satisfying
 \[
  x_{i}\ll [u_{2i-1}] \ll [u_{2i}] \ll x_{i+1},\quad
  [v_{2i}]\leq x_{i+1}',
 \]
and 
\[
 u_{1}\ll\cdots\ll u_{2n-2},\andSep u_{i}\propto v_{i}.
\]

Since $S_{k}$ has refinable sums, we obtain decreasing sequences $(y_{j}^{i})_{j=1}^{l}$ for $i\leq 2n-2$ satisfying the properties of \autoref{dfn_Rs}. In particular, we get
\[
 x_{i}\ll [u_{2i-1}]\ll [y_{1}^{2i-1}]+\cdots +[y_{1}^{2i-1}]\ll [u_{2i}] \ll x_{i+1},
\]
and $[y_{1}^{2i-1}]\propto v_{i+1}\leq x'_{i+1}$.

It follows that $S$ has refinable sums.
\end{proof}
%===========================================

\begin{dfn}\label{dfn_almost_ordered}
 A \CuSgp{} $S$ is said to have \emph{almost ordered sums} if for any finite set of elements $x_{1},\cdots ,x_{n}$ in $S$ there exists elements $y_{j,i}$ in $S$ such that
 \[
  x_{1}+\cdots +x_{n}=\sup_{i} (y_{1,i}+\cdots +y_{n,i})
 \]
and such that
\begin{enumerate}[(i)]
 \item $y_{1,i}\geq\cdots \geq y_{n,i}$
 \item $(y_{n,i})_{i}$ is increasing and bounded by $x_{1},\cdots ,x_{n}$.
 \item If $x'\ll x_{j_{1}},\cdots ,x_{j_{r}}\leq z$ for $j_{1},\cdots ,j_{r}$ pairwise different, we have $x'\leq y_{r,i}$ and $y_{n+1-r,i}\leq z$ for every sufficiently large $i$.
\end{enumerate}
\end{dfn}
%===========================================

\begin{exa}\label{exa_ord}
 If $S$ is a distributively lattice ordered \CuSgp{}, $S$ has almost ordered sums. This applies, in particular, to \CuSgp{s} such that $S\cong\Lsc (X,\NNbar )$ for some $X$.
  
Indeed, given $x_{1},\cdots ,x_{n}$, set
\[
  y_{1,i} = x_{1}\vee\cdots\vee x_{n},\quad 
  y_{2,i} = (x_{1}\wedge x_{2})\vee \cdots\vee (x_{n-1}\wedge x_{n}),\quad 
  \cdots ,\andSep 
  y_{n,i} =x_{1}\wedge\cdots\wedge x_{n}
\]
for every $i$, and note that by \autoref{std_formula} we have
\[
 x_{1}+\cdots +x_{n}=y_{1,i}+\cdots +y_{n,i}.
\]

This implies that $S$ has almost ordered sums.
\end{exa}
%========================================

\begin{exa}
 Let $Z'=Z\cup\{ 1''\}$ with $1''$ a compact element not comparable with $1$ such that   $1+x=1''+x$ for every $x\in Z\setminus\{ 0\}$ and $k1''=k$ for every $k\in\NN$. Then, $Z'$ does not have almost ordered sums.
 
 To see this, consider the sum $1+1''$ and assume, for the sake of contradiction, that $Z'$ has refinable sums. Then, there exist elements $y_{1,i},y_{2,i}$ such that $1+1''=\sup_{i}y_{1,i}+y_{2,i}$.
 
 Since $1+1''=2$ is compact, for every sufficiently large $i$ we have $1+1''= y_{1,i}+y_{2,i}$.
 
 This implies that $y_{1,i}=2$ and $y_{2,i}=0$, since we know that $1,1''\leq y_{1,i}$ and that $1$, $1''$ are not comparable.
 
 However, we also have $1,1''\leq 1.5$, so we get $2=y_{1,i}\leq 1.4$, a contradiction.
\end{exa}
%===========================================

\begin{prp}\label{prp:AlmOrdSum}
Inductive limits of distributively lattice ordered \CuSgp{s} have almost ordered sums.
\end{prp}
\begin{proof}
 Let $S=\lim_{k} (S_{k},\varphi_{k+1,k})$ be the inductive limit of distributively lattice ordered \CuSgp{s} $S_{k}$. As before, given an element $x\in S_k$ let us denote its image through the canonical map $S_k\to S$ by $[x]$.
 
 Let $x_{1},\cdots ,x_{n}$ be elements in $S$. One can check that there exists an increasing sequence of integers $(k_{l})$ and elements take $x_{1}^{l},\cdots ,x_{n}^{l}\in S_{k_{l}}$ such that $([x_{j}^{l}])_{l}$ are $\ll$-increasing sequences in $S$ with suprema $x_{j}$ for every $j\leq n$, in such a way that $\varphi_{k_{l+1},k_{l}}(x_{j}^{l})\ll x_{j}^{l+1}$ for every $j$ and $l$.
 
 Since each $S_{k}$ is distributively lattice ordered, for every $l$ there exist elements $y_{1}^{l},\cdots ,y_{n}^{l}$ in $S_{k_{l}}$ with 
 \[
 x_{1}^{l}+\cdots +x_{n}^{l}=y_{1}^{l}+\cdots +y_{n}^{l}
 \]
 satisfying the properties of \autoref{dfn_almost_ordered} (see \autoref{exa_ord} above). This implies, in particular, $\sup_{l}([y_{1}^{l}]+\cdots + [y_{n}^{l}])=x_{1}+\cdots +x_{n}$.

We will now check that the elements $[y_{1}^{l}],\cdots , [y{n}^{l}]$ satisfy conditions (i)-(iii) in \autoref{dfn_almost_ordered}:

By construction, one has $y_{1}^{l}\geq\cdots\geq y_{n}^{l}$ for every $l$, so condition (i) is satisfied. For condition (ii), let $l\in\NN$.  Then, applying condition (ii) in $S_{k_{l}}$, we have
\[
 \varphi_{k_{l+1},k_{l}} (y_{n}^{l})
 \leq 
 \varphi_{k_{l+1},k_{l}}(x_{1}^{l}), \cdots , \varphi_{k_{l+1},k_{l}}(x_{n}^{l})\ll x_{1}^{l+1},\cdots, x_{n}^{l+1}
\]
and, by condition (iii) in $S_{k_{l+1}}$, we get $\varphi_{k_{l+1},k_{l}} (y_{n}^{l})\leq y_{n}^{l+1}$. It follows that condition (ii) is satisfied.

To prove (iii), take $x',z\in S$ such that $x'\ll x_{j_{1}}, \cdots ,x_{j_{r}}\leq z$ for some pairwise different $j_{1},\cdots ,j_{r}\leq n$. For a large enough $l$, there exist $u$ such that
\[
 u\ll x_{j_{1}}^{l},\cdots ,x_{j_{r}}^{l}
\]
and $x'\ll [u]$ in $S$.

This implies that, for every $l'\geq l$, one has
\[
 \varphi_{k_{l'},k_{l}}(u)\ll \varphi_{k_{l'},k_{l}}(x_{j_{1}}^{l})+\cdots +
 \varphi_{k_{l'},k_{l}}(x_{j_{r}}^{l})
\]
in $S_{k_{l'}}$.

Consequently, we have $\varphi_{k_{l'},k_{l}}(u)\leq d_{r}^{l'}$ and so $x\ll [u]\ll [d_{r}^{l'}]$ in $S$ for every $l'\geq l$. Also, since $x_{j_{1}}, \cdots ,x_{j_{r}}\leq b$, for every $l$ there exists some $z_{l}\in S_{k_{l}}$ with $[z_{l}]\ll z$ and 
\[
 x_{j_{1}}^{l},\cdots ,x_{j_{r}}^{l}\leq z_{l}.
\]

Therefore, one gets $y_{n+1-r}^{l}\leq z_{l}$. This implies $[y_{n+1-r}^{l}]\leq [z_{l}]\ll z$ for every $l$, as required.
\end{proof}
%===========================================

\begin{thm}\label{thm:AIProperties}
 Let $A$ be an AI-algebra. Then, its Cuntz semigroup $\Cu (A)$ is weakly chainable and has refinable sums and almost ordered sums.
\end{thm}
\begin{proof}
The Cuntz semigroup 
 $\Cu (A)$ is weakly chainable by \autoref{cor:WeakCha}. Further, using the same arguments as in \autoref{lma:WeakChSums}, it is easy to see that finite direct sums of \CuSgp{s} having refinable sums or almost ordered sums have refinable sums or almost ordered sums respectively. Thus, it  follows from \autoref{exa:LscRefSum} and \autoref{prp:RefSum} that $S$ has refinable sums.
 
 By \autoref{exa_ord} and \autoref{prp:AlmOrdSum}, $\Cu (A)$ also has almost ordered sums.
\end{proof}

%===========================================
%===========================================
%===========================================

\bibliographystyle{aomalphaMyShort}
\bibliography{ReferencesCommCuAI}

\end{document}